\documentclass[12pt,oneside,reqno]{amsart}

\usepackage[a4paper,left=25mm,top=31mm,right=25mm,bottom=30mm]{geometry}

\usepackage{amssymb,amsfonts}
\usepackage[all,arc]{xy}
\usepackage{enumerate}
\usepackage{siunitx}
\usepackage{enumitem}
\usepackage{mathtools}
\usepackage{mathrsfs}
\usepackage{dsfont}
\usepackage{amsfonts}
\usepackage{amssymb}
\usepackage{graphicx}
\usepackage{ragged2e}
\usepackage{amsthm}
\usepackage{amsmath}
\usepackage[mathscr]{euscript}
 \let\mathscr\relax
\usepackage[scr]{rsfso}
\usepackage{graphicx}
\usepackage{color}
\usepackage{float}
\usepackage{caption}
\usepackage[pdftex,bookmarksnumbered,bookmarksopen,linkcolor=blue,colorlinks=true, citecolor=blue]{hyperref}

\usepackage{mathtools,xparse}

\DeclareMathOperator{\vol}{Vol}
\DeclareMathOperator{\Var}{Var}

\newcommand {\E} {\mathbb{E}}
\newcommand {\R} {\mathbb{R}}

\newtheorem{thm}{Theorem}[section]
\newtheorem{cor}[thm]{Corollary}
\newtheorem{prop}[thm]{Proposition}
\newtheorem{lem}[thm]{Lemma}

\theoremstyle{definition}
\newtheorem{defn}[thm]{Definition}

\theoremstyle{remark}
\newtheorem{rem}[thm]{Remark}

\numberwithin{equation}{section}

\newcommand{\Z}{\mathbb{Z}}

\begin{document}
	\title{Expected nodal volume for non-Gaussian random band-limited functions}

	\author[Z. Kabluchko]{Zakhar Kabluchko}
	\address{Zakhar Kabluchko: Institut f\"ur Mathematische Stochastik,
		Westf\"alische Wilhelms-Universit\"at M\"unster, Germany}
	\email{zakhar.kabluchko@uni-muenster.de}

	\author[A. Sartori]{Andrea Sartori}
	\address{School of Mathematical Sciences, Tel Aviv University, Israel}
	\email{andrea.sartori.16@ucl.ac.uk}

	\author[I. Wigman]{Igor Wigman}
	\address{Department of Mathematics, King's College London, UK}
	\email{ igor.wigman@kcl.ac.uk}

	\maketitle
	\begin{abstract}
		The asymptotic law for the expected nodal volume of random non-Gaussian monochromatic band-limited functions is determined in vast generality. Our methods combine microlocal analytic techniques and modern probability theory. A particularly challenging obstacle needed to overcome is the possible concentration of nodal volume on a small proportion of the manifold, requiring solutions in both disciplines. As for the fine aspects of the distribution of nodal volume, such as its variance, it is expected that the non-Gaussian monochromatic functions behave qualitatively differently compared to their Gaussian counterpart, with some conjectures been put forward.
	\end{abstract}
	\section{Introduction}
	\label{intro}
	In the recent years a lot of effort has been put into understanding
	the geometry of Laplace eigenfunctions on smooth manifolds.
	Let $(M,g)$ be a smooth compact Riemannian manifold of dimension $n$, and $\Delta=\Delta_{g}$ the Laplace-Beltrami operator on $M$.
	Denote $\{\lambda_{i}\}_{i\ge 1}$ to be the (purely discrete) spectrum of $\Delta$, with the corresponding Laplace eigenfunctions
	$\phi_{i}$ satisfying
	\begin{equation*}
		\Delta \phi_{i}+\lambda_{i}^{2}\phi_{i} = 0.
	\end{equation*}
	An important qualitative descriptor of the geometry of $\phi_{i}$ is its {\em nodal set} $\phi_{i}^{-1}(0)$, and, in particular, the nodal volume $\mathcal{V}(\phi_{i})=\mathcal{H}^{n-1}(\phi_{i}^{-1}(0))$, that is the $(n-1)$-dimensional Hausdorff measure of $\phi_{i}^{-1}(0)$.
	
	The highly influential {\em Yau's conjecture} \cite{Ysurvey} asserts that, in the high energy limit $\lambda_{i}\rightarrow\infty$, the nodal volume of $\phi_{i}$ is commensurable with $\lambda_{i}$: there exists constants $C_{M}>c_{M}>0$ so that
	$$c_{M}\cdot \lambda_{i}\le \mathcal{V}(\phi_{i}) \le C_{M}\cdot \lambda_{i}.$$ Yau's conjecture was established for the real analytic manifolds \cite{B78,BG72,DF}, whereas, more recently, the optimal lower bound and polynomial upper bound were proved \cite{L2,L1,LM} in the smooth case.
	
	\vspace{2mm}
	
	In his seminal work \cite{B1}
	Berry proposed to compare the (deterministic) Laplace eigenfunctions on manifolds, whose geodesic flow
	is ergodic, to the \textit{random} monochromatic isotropic waves, that is, a Gaussian stationary isotropic random field
	$F_{\mu}:\mathbb{R}^{n}\rightarrow\mathbb{R}$, whose spectral measure $\mu$ is the hypersurface measure on the sphere $\mathbb{S}^{n-1}\subseteq \mathbb{R}^{n}$,
	normalized by unit total volume.
	Equivalently, $F_{\mu}(\cdot)$ is uniquely defined via its covariance function
	\begin{equation}
		\label{eq:covar mono}
		K_{\infty}(x,y):=\mathbb{E}[F_{\mu}(x)\cdot F_{\mu}(y)] = \int\limits_{\mathbb{S}^{n-1}}e^{2\pi i \langle x-y,\xi\rangle}d\mu(\xi)
	\end{equation}
	For example, in $2d$, the covariance function of $F_{\mu}:\mathbb{R}^{2}\rightarrow\mathbb{R}$ is given by
	$$\mathbb{E}[F_{\mu}(x)\cdot F_{\mu}(y)] = J_{0}(|x-y|),$$ with $J_{0}$ the Bessel $J$ function of order $0$. Berry's conjecture should be understood 	in some random sense, e.g. when averaged over the energy level. Alternatively, one can consider some random ensemble of eigenfunctions or their random linear combination, Gaussian or non-Gaussian.
	
	A concrete ensemble of the said type is that of {\em band-limited} functions \cite{SW}
	\begin{equation}
		\label{function}
		f_{T}(x)= f(x) = v(T)^{-1/2}\sum\limits_{\lambda_{i}\in [T-\rho(T),T]}a_{i}\phi_{i}(x),
	\end{equation}
	where $a_{i}$ are centred unit variance i.i.d. random variables, Gaussian or non-Gaussian,
	$T\rightarrow\infty$ is the {\em spectral parameter},
	and the summation is over the {\em energy window} $[T-\rho(T),T]$ of width $\rho=\rho(T)\geq 1$.  The
	convenience pre-factor, which has no impact on the nodal structure of $f_{T}(\cdot)$,
	\begin{equation}
		\label{defv}
		v(T):=\frac{(2\pi)^{n}}{\omega_n\cdot\vol(M)}\rho(T)T^{n-1}=c_{M}\rho(T)T^{n-1},
	\end{equation}
	where $\omega_{n}$ is the volume of the unit ball in $\mathbb{R}^{n}$, represents the asymptotic, as $T\rightarrow \infty$, number of summands\footnote{If $M$ is periodic and $\rho(T)=O(1)$ we assume that $\rho(T)$ is sufficiently large, depending on $M$, so that to avoid trivialities. Alternatively, one may assume that $T$ belongs to a particular subsequence, see \eqref{eq:choice T} below.} in \eqref{function} so that $f_{T}(x)$ is of asymptotic unit variance at each $x\in M$.
	
	\vspace{2mm}
	
	Regardless of whether or not $f_{T}(\cdot)$ in \eqref{function} is Gaussian,
	its covariance kernel is the
	function $K_{T}:M\times M \rightarrow \mathbb{R}$ given by
	\begin{equation}
		\label{eq:KT covar band-limited}
		K_{T}(x,y):= \mathbb{E}[f_{T}(x)\cdot f_{T}(y)] = \frac{1}{v(T)}\sum\limits_{\lambda_{i}\in [T-\rho(T),T]}\phi_{i}(x)\cdot \phi_{i}(y),
	\end{equation}
	coinciding with the \textit{spectral projector} in $L^{2}(M)$ onto the subspace spanned by the eigenfunctions
	$\{\phi_{i}\}_{\lambda_{i}\in [T-\rho(T),T]}$ (recall that $a_{i}$ are unit variance).
	In what follows we will focus on the most interesting (and, in some aspects, most difficult) {\em monochromatic} regime
	$\rho(T)=o_{T\rightarrow\infty}(T)$ (``short energy window").
	Our main result (Theorem \ref{main thm}) along with all our arguments remain valid for the case
	$$\rho \underset{T\rightarrow\infty}{\sim} \alpha\cdot T,$$ with some $\alpha\in [0,1]$,
	except that the limit random field is different, resulting in a different, but explicit, constant.

	It is well-known that, under suitable assumptions on $M$ (explicated below), and assuming, as we did, that
	$\frac{\rho(T)}{T}\underset{T\rightarrow\infty}{\longrightarrow}0$ sufficiently slowly,
	the covariance \eqref{eq:KT covar band-limited}, after scaling by $\sqrt{T}$, is asymptotic to \eqref{eq:covar mono}, around
	every (or almost every) reference point $x$, in the following sense. Let $\exp_{x}:T_{x}M\rightarrow M$ be the exponential map,
	that is a bijection between a ball $B(r)\subseteq \R^{n}$ centred at $0\in \mathbb{R}^n$ and some neighbourhood in $M$ of $x$, with $r>0$ depending only on $M$,
	independent of $x\in M$. Then, for every $R>0$,
	\begin{equation}
		\label{eq:KT->Kinf}
		K_{T}\left(\exp_{x}(x/T),\exp_{x}(y/T)\right) \underset{T\rightarrow\infty}{\longrightarrow} K_{\infty}(x,y),
	\end{equation}
	uniformly for $\|x\|,\|y\|\le R$, with $K_{\infty}(\cdot,\cdot)$ as in \eqref{eq:covar mono},
	with the convergence \eqref{eq:KT->Kinf} holding together with an arbitrary number of derivatives \cite{CH15,CH18,SVbook}.
	Since in the Gaussian case, the finite-dimensional distributions are determined by the covariances, if $f_{T}(\cdot)$ is Gaussian, the convergence \eqref{eq:KT->Kinf} readily implies that the law
	of $f_{T}$ locally converges to that of $F_{\mu}$. For the non-Gaussian case, a significant proportion of our argument below will exploit the the increasing number of summands in \eqref{function} to apply the Central Limit Theorem together with the convergence \eqref{eq:KT->Kinf} to yield the local convergence of the law of $f_{T}$ to that of $F_{\mu}$.
	
	For the linear combinations
	\eqref{function} of Laplace eigenfunctions on real analytic $M$, the deterministic upper bound analogue
	\begin{equation*}
		\mathcal{V}(f_{T}) \le C_{M}\cdot T
	\end{equation*}
	of Yau's conjecture remains valid, thanks to the work of Jerison-Lebeau \cite[Section 14]{DIbook}.	
	The principal result of this manuscript determines the
	precise asymptotic growth of the expected nodal volume of monochromatic random band-limited functions on generic real analytic manifolds with no boundary, under the mere assumptions that $a_{n}$ have finite second moment:

	\begin{thm}
		\label{main thm}
		Let $(M,g)$ be a real analytic compact $n$-manifold with empty boundary,
		assume that $\rho(T)=o_{T\rightarrow\infty}(T)$, and either $n\ge 3$ and $\rho(T)\ge 1$, or $n=2$ and $\rho(T)\ge \frac{T}{\log{T}}$,
		and let $f_{T}(\cdot)$ be the band-limited functions \eqref{function} with $a_{n}$ centred i.i.d. so that $\mathbb{E}[a_{n}^{2}]=1$.
		Then
		\begin{equation}
			\label{eq:main res}
			\E[\mathcal{V}(f_{T})] = \vol(M)\left(\frac{4\pi}{n} \right)^{1/2}  \frac{\Gamma\left( \frac{n+1}{2} \right)}{\Gamma\left( \frac{n}{2}  \right)}T
			+o_{T\rightarrow\infty}(T).
		\end{equation}
	\end{thm}
	
	\subsection*{Some conventions}
	We write $A\lesssim B$ to designate the existence of some constant $C>0$ such that $A\leq C B$. We also write $C,c>0$ for constants whose value may change from line to line. Every constant implied in the notation may depend on the pair $(M,g)$, but, since $(M,g)$ is fixed, we suppress this dependence in the notation. We also write $B(x,r)$ for the (Euclidean) ball centred at $x$ of radius $r>0$, $B_g(\cdot)$ for the geodesic ball and,  for brevity, we write $B_0=B(0,1)\subset \mathbb{R}^n$.  Given a ball $B$ and some $r>0$, we write $\overline{B}$ for its closure and $rB$ for the concentric ball with $r$-times the same radius. Moreover, we use the multi-index notation $D^{\alpha}= \partial_{x_1}^{\alpha_1}...\partial_{x_2}^{\alpha_2}$ where $\alpha=(\alpha_1,...,\alpha_n)$ and $|\alpha|=\alpha_1+...+\alpha_n$. Furthermore, given a ($C^3$) function $g:B(x,r)\rightarrow \mathbb{R}$ and some $r>0$, we let $$\mathcal{V}(g,B(x,r))= \mathcal{H}^{n-1}\{x\in B(x,r): g(x)=0\}$$
	be the nodal volume of $g$ in $B(x,r)$. Since it will be often useful to change scales, we also write $g_r(\cdot)=g(r\cdot)$ and notice that
	$$\mathcal{V}(g_r,B(1))r^{n-1}= \mathcal{V}(g, B(r)).$$
	Finally, we denote by $\Omega$ the abstract probability space where every random object is defined.

	\subsection*{Acknowledgements}
	
	We would like to thank S. Zelditch for useful discussions, and, in particular, for sharing with us his unpublished results on self-focal points on analytic manifolds, demonstrating that the class of real analytic manifolds, for which Theorem \ref{main thm} applies, is unrestricted, see the discussion in section \ref{sec:self-focal pnts}. A. Sartori was supported by the Engineering and Physical
	Sciences Research Council [EP/L015234/1], ISF Grant 1903/18 and the IBSF Start up Grant no. 2018341.
	Z. Kabluchko was supported by the German Research Foundation under Germany's Excellence Strategy  EXC 2044 -- 390685587, Mathematics M\"unster: Dynamics - Geometry - Structure.
	
	\section{Discussion}
	
	\subsection{Survey of non-Gaussian literature}
	
	To our best knowledge, Theorem \ref{main thm} is the first universality result applicable in the asserted vastly general scenario, in terms of both the underlying manifold $M$ and the random coefficients $\{a_{i}\}$. Our approach is based on a blend of microlocal analytic techniques,
	missing from the existing non-Gaussian literature, and purely probabilistic methods. The closest analogue to Theorem \ref{main thm} we are aware of in the existing literature is \cite{APP18}, dealing with $2$d random non-Gaussian trigonometric polynomials: these are related to the
	random band-limited Laplace eigenfunctions on the standard $2$d torus, corresponding to the long energy window $\rho(T)=T$ (here, the energies ordering
	is somewhat different, to allow for separation of variables). The asymptotics for the expected nodal length was asserted for centred unit variance random variables, in perfect harmony to \eqref{eq:main res} (though with a different leading constant, a by-product of a non-monochromatic scaling limit).
	
	Even though we didn't meticulously validate all the details, we believe that their arguments translate verbatim for the ``pure" $2$d toral Laplace eigenfunctions
	\begin{equation}
		\label{eq:fn arw def}
		g_{n}(x) = \sum\limits_{\underset{\|\mu\|^{2}=n}{\mu\in\mathbb{Z}^{2}}} a_{\mu}\cdot e(\langle \mu,x\rangle),
	\end{equation}
	where the $a_{\mu}$ are i.i.d., save for the relation $a_{-\mu}=\overline{a_{\mu}}$ making $g_{n}$ real-valued, and the summation on the r.h.s. of \eqref{eq:fn arw def} is w.r.t. to all standard lattice points lying on the radius-$\sqrt{n}$ centred circle. In the Gaussian context
	the $g_{n}$ are usually referred to as ``arithmetic random waves" (ARW), see e.g. \cite{KKW,ORW08,RW08}; they are the
	band-limited functions for the standard flat torus corresponding to the ``short energy window" $\rho(T)=1$
	(in fact, in this case, the energy window width could be made infinitesimal).
	Other than the result for $2$d random trigonometric polynomials all the literature concerning real zeros of non-Gaussian ensembles is $1$-dimensional in essence: real zeros of random algebraic polynomials or Taylor series, see e.g. \cite{IM2,IM,NO} and the references therein,
	random trigonometric polynomials on the circle \cite{BCP19},
	and the restrictions of $2$d random toral Laplace eigenfunctions \eqref{eq:fn arw def} to a smooth curve \cite{CNNV20}.
	
	\subsection{Gaussian vs. non-Gaussian monochromatic functions: cases of study}
	
	Unlike the non-Gaussian state of art concerning the zeros of the band-limited functions, the Gaussian literature is vast and rapidly expanding, thanks to the powerful Kac-Rice method tailored to this case, at times, combined with the Wiener chaos expansion. Here the literature varies
	from the very precise and detailed results concerning the zero volume distribution (its expectation, variance and limit law),
	restricted to some particularly important ensembles, such as random spherical harmonics \cite{MRW20,W10} or
	the arithmetic random waves \cite{KKW,MRW16}, to somewhat less detailed results, but of far more general nature \cite{CH20,Z2009},
	to almost sure asymptotic result \cite{Gass20} w.r.t. a randomly independently drawn sequence of functions $\{f_{T}\}_{T}$.
	
	It is plausible, if not likely, that, under a slightly more restrictive assumptions on the random variables, our techniques yield a power saving upper bound for the nodal length variance of the type
	\begin{equation*}
		\Var\left(\frac{f_{T}}{T}\right) =O\left( T^{-\delta} \right)
	\end{equation*}
	for some $\delta>0$, but certainly not a {\em precise} asymptotic law for the variance, even for the particular cases of non-Gaussian random spherical harmonics or the non-Gaussian Arithmetic Random Waves. In the Gaussian case even some important {\em non-local} properties of the nodal set were addressed: the expected number of nodal components \cite{NS09,NS}, their fluctuations \cite{BMM19,NS20}, their fine topology and geometry, and their relative position \cite{BW18,SW}.
	
	\vspace{2mm}
	
	The aforementioned random ensemble of Gaussian spherical harmonics is the sequence of functions $f_{\ell}:\mathbb{S}^{2}\rightarrow\R$, $\ell\ge 1$, where
	\begin{equation}
		\label{eq:fl spher har}
		f_{\ell}(x) = \frac{1}{\sqrt{2\ell+1}}\sum\limits_{m=-\ell}^{\ell}a_{\ell,m}Y_{\ell,m}(x),
	\end{equation}
	with $\{Y_{\ell,m}\}_{-\ell\le m\le \ell}$ the standard basis of degree-$\ell$ spherical harmonics, and $a_{\ell,m}$ i.i.d. standard Gaussian
	random variables. An application of the Kac-Rice formula yields \cite{B85} the expected nodal length of $f_{\ell}(\cdot)$ to be given precisely by $$\E[\mathcal{V}(f_{\ell})] = \sqrt{2}\pi\cdot \sqrt{\ell(\ell+1)}\sim \sqrt{2}\pi\ell,$$ whereas a significantly heavier machinery, also appealing to the Kac-Rice method, yields \cite{W10} a precise asymptotic law
	\begin{equation*}
		\Var(\mathcal{V}(f_{\ell})) \underset{\ell\rightarrow\infty}{\sim} \frac{1}{32}\log{\ell},
	\end{equation*}
	smaller than what would have been thought the natural scaling $\approx c\cdot \ell$ would be (``Berry's cancellation phenomenon").
	
	For the non-Gaussian random spherical harmonics, Theorem \ref{main thm} is not directly applicable, because of the extra condition
	$\rho(T)\ge \frac{T}{\log{T}}$ in $2$d.
	However, in Appendix \ref{sec:sphere} below we were able to extend the validity of Theorem \ref{main thm} to this important ensemble, at least for Bernoulli random variables (see Theorem \ref{thm sphere}).
	In light of the non-universality result of ~\cite{BCP19}, it is {\em not unlikely} that in the non-Gaussian case (i.e. the $a_{\ell,m}$ are centred unit variance i.i.d. random variable), the variance satisfies the $2$-term asymptotics
	\begin{equation*}
		\Var(\mathcal{V}(f_{\ell})) =c_{1}\cdot \ell+c_{2}\cdot \log{\ell}+O(1),
	\end{equation*}
	with $c_{1},c_{2}$ depending on the law of $a_{\ell,m}$ and $c_{1}$ vanishing for a peculiar family of distributions, including the Gaussian one.
	It seems less likely, though {\em conceivable}, that $c_{1}\equiv 0$.
	
	\vspace{2mm}
	
	For the $2$d Gaussian arithmetic random waves \eqref{eq:fn arw def}, it was found that the expected nodal length is given precisely by $\E[\mathcal{V}(g_{n})]=\frac{\pi}{\sqrt{2}}\cdot \sqrt{n}$, whereas the variance is asymptotic to
	$$\Var(\mathcal{V}(g_{n}))\sim 4\pi^{2}  b_{n}\cdot   \frac{n}{r_{2}(n)^{2}},$$ where $r_{2}(n)$ is
	the number of summands in \eqref{eq:fn arw def}. Here the
	numbers $b_{n}$ are
	genuinely fluctuating in $[1/512,1/256]$, depending on the angular distribution of the lattice points in the summation on the r.h.s.
	of \eqref{eq:fn arw def}, and the leading term corresponding to $\frac{n}{r_{2}(n)}$ ``miraculously" cancelling out precisely (``arithmetic Berry's cancellation").
	
	Using the same reasoning as for the spherical harmonics, for the non-Gaussian case (i.e. $a_{\mu}$ are centred unit variance i.i.d. random variables), it is {\em expected} that the $2$-term asymptotics
	$$\Var(\mathcal{V}(g_{n}))\sim \widetilde{c_{1}}\frac{n}{r_{2}(n)} + \widetilde{c_{2}}\frac{n}{r_{2}(n)^{2}}$$ holds with $\widetilde{c_{1}},\widetilde{c_{2}}$ possibly depending on both the law of $a_{\mu}$ and the angular distribution of the lattice points $\{\mu\}$ in \eqref{eq:fn arw def},
	with $\widetilde{c_{1}}$ vanishing for $a_{\mu}$ a peculiar class of distribution laws, including the Gaussian (whence $\widetilde{c_{1}}$ vanishes independent of the angular distribution of the lattice points $\{\mu\}$). The dependence of $\widetilde{c_{1}}$ and $\widetilde{c_{2}}$ on both the distribution law of $a_{\mu}$ and the angular distribution of $\{\mu\}$ is of interest, in particular, whether the vanishing of $\widetilde{c_{1}}$ depends on the angular distribution of $\{\mu\}$ at all (which is not the case if $a_{\mu}$ is Gaussian). Again, it is {\em conceivable} that $\widetilde{c_{1}}\equiv 0$. We leave all of the above questions to be addressed elsewhere.
	
	\subsection{Self-focal points}
	
	\label{sec:self-focal pnts}
	
	An earlier version of this manuscript contained a version of Theorem \ref{main thm}, applicable under a seemingly somewhat restrictive
	(though very mild) assumption on $M$, rather than its mere analyticity, concerning its so-called {\em self-focal} points. It was pointed to us by S. Zelditch, that the extra assumption
	is redundant, as explained below, after a few necessary definitions.
	
	\begin{defn}
		\label{focal and twisted}
		Let $(M,g)$ be a smooth compact manifold without boundaries, $S^{*}M$ the cotangent sphere bundle on $M$,
		and $G^{t}:S^{*}M\rightarrow S^{*}M$ the geodesic flow on $M$.
		
		\begin{enumerate}
			
			\item The set of loop directions based at $x$ is
			\begin{equation*}
				\mathcal{L}_{x} = \{ \xi \in S^{*}_{x}M: \: \exists t>0.\, \exp_{x}(t\xi) = x\}.
			\end{equation*}

			\item The set of closed geodesics based at $x$ is
			\begin{equation*}
				\mathcal{CL}_{x} = \{ \xi \in S^{*}_{x}M: \: \exists t>0.\, G^{t}(x,\xi)=(x,\xi)\}.
			\end{equation*}
			
			\item A point $x\in M$ is {\bf self-focal},
			if $|\mathcal{L}_{x}|>0$, where $|\cdot |$ is the natural measure on $S^{*}_{x}$ induced by
			the metric $g_{x}(\cdot, \cdot)$.
			
			\item The geodesic flow on $M$ is {\bf periodic}, if the set of its closed geodesics is of full Liouville measure
			in $S^{*}M$. The geodesic flow on $M$ is {\bf aperiodic} if the set of its periodic closed geodesics is of Liouville measure $0$.

		\end{enumerate}		
		
	\end{defn}
	
	We observe that for $M$ real analytic, the set of its periodic geodesics, is of either full or $0$ Liouville measure in $S^{*}M$ (see either \cite[Lemma 1.3.8]{SVbook} or Lemma \ref{lem: aper or per} below). Hence, in the real analytic case, {\bf the geodesic flow on $M$ is either periodic or aperiodic}.
	
	\vspace{2mm}
	
	Originally, Theorem \ref{main thm} assumed that if $M$ is aperiodic, then the set of its self-focal points is of measure $0$. However, it was demonstrated ~\cite{Zelditchpriv} that if $M$ is aperiodic, then the set of self-points has automatically measure $0$, i.e. the said extra assumption is redundant, see also section \ref{local Weyl's law} below.

	\section{Preliminaries}
	
	\subsection{Geodesic flow and the spectrum of $\sqrt{-\Delta}$ on $M$}
	\label{geodesic flow}
	Let $T^{*}M$ and  $S^{*}M$ be the co-tangent and the co-sphere bundle on $M$ respectively. The geodesic flow
	\begin{align}
		\label{def geodesic flow}
		G^t: T^{*}M \rightarrow T^{*}M
	\end{align}
	is the Hamiltonian flow of the metric norm function
	\begin{align} \nonumber& H: T^{*}M \rightarrow \mathbb{R} &H(x,\xi)= \sum_{i,j=1}^n g^{ij}\xi_i\xi_j,
	\end{align}
	where $g=g_{ij}$ is the metric on $M$ and $g^{ij}$ is its inverse. Since $G^t$ is homogeneous, from now on, we will consider only its restriction to $S^{*}M$.
	We will need the following simple lemma, see also \cite[Lemma 1.3.8]{SVbook}
	\begin{lem}
		\label{lem: aper or per}
		If $(M,g)$ is a real analytic manifold, then the set of closed geodesics, on the co-sphere bundle equipped with the Liouville measure,  has either full measure or measure zero.
	\end{lem}
	\begin{proof}
		Since closed geodesics are defined by
		$$ G^t(x,\xi)= (x,\xi),$$
		the set of closed geodesic, for fixed $t>0$ is the zero set of an analytic function and therefore it must have co-dimension at least $1$ or be trivial.
	\end{proof}
	Lemma \ref{lem: aper or per} implies that the geodesic flow, on a real analytic manifold, is either \textit{aperiodic} if the set of closed geodesics has measure zero, or \textit{periodic} with (minimal) period $H>0$ if $G^H=id$. In the latter case the manifold is also called \textit{Zoll}. Therefore, we have the following description of the spectrum of $\sqrt{-\Delta}$ of real-analytic $(M,g)$, see \cite[Theorem 8.4]{Zbook} and references therein.
	
	
	Suppose that the geodesic flow on $M$ is \textit{aperiodic}.
	The two-term Weyl law of Duistermaat-Guilleimin(-Ivrii)
	states $$ |\{i>0: \lambda_{i}\leq T\}|= c_MT^{n} + o(T^{n-1}).$$
	
	Now assume that the geodesic flow on $M$ is \textit{periodic} with period $H$.
	Then the spectrum of $\sqrt{\Delta}$ is a union of clusters of the form
	$$ C_k:=\left\{\frac{2\pi}{H}\left( k+ \frac{\beta}{4}\right) + \mu_{k_i} \hspace{4mm} \text{for} \hspace{4mm} i=1,...,d_k\right\} \hspace{5mm} k=1,2...,$$
	where  $\mu_{k_i}= O(k^{-1})$ uniformly for all $i$, $d_k$ is a polynomial in $k$ of degree $n-1$ and $\beta$ is the Morse index of $M$. In particular, in order to avoid trivial summation in \eqref{function}, if the geodesic flow is periodic and $\rho(T)=O(1)$, we will assume that either $T=T(k)$ is of the form
	\begin{align}
		\label{eq:choice T}
		T=\frac{2\pi}{H}\left( k+ \frac{\beta}{4}\right)+1,
	\end{align}
	or, alternatively, that $\rho(T)$ is sufficiently large.
	
	We will need the following Lemma, see \cite[Proposition 2.3 (A)]{Z2009}.
	\begin{lem}
		\label{diagonal Weyl}
		Let $(M,g)$ be a real analytic compact manifold without boundary of dimension $n$. Then, uniformly for all $x\in M$, for $c_M$ as in \eqref{defv},  we have
		$$\sum_{\lambda_i\in [T-\rho(T),T]} |\phi_{i}(x)|^2= c_M \rho(T)T^{n-1}(1+o_{T\rightarrow \infty}(1)).$$
	\end{lem}
	\subsection{Local Weyl's law}
	\label{local Weyl's law}
	To state the main result of this section, we first need to introduce some notation. Let $x\in M$ and let $F_x$ be $f$ rescaled to the ball $B_g(x,1/T)$ in normal coordinates. More precisely, following \cite{NS}, we define:
	\begin{align}
		F_{T,x}(y)=F_x(y)= f(\exp_x(y/T)) \label{def of F_x}
	\end{align}
	for $y\in B(0,1)=:B_0\subset \mathbb{R}^n$, where $\exp_x: \mathbb{R}^n\cong T_xM \rightarrow M$ is the exponential map. Notice that, in the definition of the exponential map, we have tacitly identified, via an Euclidean isometry, $\mathbb{R}^n$ with $T_xM$. Moreover, we observe that, since $(M,g)$ is analytic, the injectivity radius of $M$ is uniformly bounded from below \cite{C70}, thus, from now on, we assume that $1/T$ is smaller than the injectivity radius so that the exponential map is a diffeomorphism.  Furthermore, thanks to \cite[Section 8.1.2]{NS} due to Nazarov and Sodin, see also \cite[Section 2]{Roz17}, the map:
	$$(x,\omega)\in M\times \Omega \rightarrow F_x(\omega, \cdot)\in C^{\infty}(B_0)$$
	is measurable.

	The main result of this section is the following:
	\begin{prop}
		\label{covariance function}
		Let $F_x$ be as in \eqref{def of F_x} and $M$ be a compact, real-analytic manifold with empty boundary. Then, for $x\in M$ outside of a measure $0$ set, one has
		\begin{align}
			\nonumber
			\mathbb{E}[F_x(y)F_x(y')]= \int_{|\xi|=1}e\left(\langle y-y', \xi\rangle\right)d\xi +o(1)= \frac{J_{\Lambda}(|y-y'|)}{|y-y'|^{\Lambda}} +o_{T\rightarrow \infty}(1)
		\end{align}
		where $\Lambda=(n-2)/2$ and $ J_{\Lambda} (\cdot)$ is the $\Lambda$-th Bessel function, uniformly for all $y,y'\in B_0$. Moreover, we can also differentiate both sides any arbitrary finite number of times, that is
		\begin{align}
			\mathbb{E}[D^{\alpha}F_x(y)D^{\alpha'}F_x(y')]=(-1)^{|\alpha'|}(2\pi i)^{|\alpha|+|\alpha'|} \int_{|\xi|=1}\xi^{\alpha+\alpha'}e\left(\langle y-y', \xi\rangle\right)d\xi +o_{T\rightarrow \infty}(1)
			\nonumber
		\end{align}
		where $\alpha,\alpha'$ are multi-indices, and $\xi^{\alpha}=(\xi_1^{\alpha_1},...,\xi_n^{\alpha_n})$.
	\end{prop}
	In order to prove Proposition \ref{covariance function}, we will need the following fact communicated to us by S. Zelditch \cite{Zelditchpriv}:
	\begin{lem}
		\label{lem: zelditch}
		Let	$M$ be a compact, real-analytic manifold with empty boundary. If the geodesic flow on $M$ is aperiodic then the set of self-focal points, as in Definition \ref{focal and twisted}, has zero volume.
	\end{lem}
	
	In light of Lemma \ref{lem: zelditch}, in order to prove Proposition \ref{covariance function}, it is sufficient to prove the following:
	\begin{prop}
		\label{two points function short}
		Let $(M,g)$ be a compact, real-analytic manifold with empty boundary. Suppose that the geodesic flow on $M$ is periodic or $x\in M$ is not a self-focal point, as in Definition \ref{focal and twisted}, then
		\begin{align}
			\sup_{y,y'\in B_g(x,10/T)}\left|\sum_{\lambda_i\in[T- \rho(T), T]} \phi_{i}(y)\phi_{i}(y')- (2\pi)^{-n}c_M T^{n}  \mathcal{J}_{\Upsilon(T)}(Td_g(y,y')) \right| = o(T^{n-1}) \label{local weyl}
		\end{align}
		where $d_g(y,y')$ is the geodesic distance between $y,y'$, $c_M>0$ is given in \eqref{defv},  $\Upsilon(T)=1-\frac{\rho(T)}{T}$ and
		\begin{align}
			\mathcal{J}_{\Upsilon(T)}(w)= \int_{\Upsilon(T) \leq |\xi|\leq 1}e(\langle w,\xi \rangle)d\xi. \label{2.2}
		\end{align}
		Moreover, we can also differentiate both sides of \eqref{local weyl} any arbitrary finite number of times, that is,
		\begin{equation*}	
			\begin{split}
				\sup_{y,y'\in B_g(x,10/T)}\frac{\left|\sum_{\lambda_i\in[T- \rho(T), T]} D^{\alpha}_{y}\phi_{i}(y)D^{\alpha'}_{y'}\phi_{i}(y')- \frac{c_M T^{n}  D^{\alpha}_{y}D^{\alpha'}_{y'}\mathcal{J}_{\Upsilon(T)}(Td_g(y,y'))}{(2\pi)^{n}} \right|}{T^{|\alpha|+|\alpha'|}} = o(T^{n-1})
			\end{split}
		\end{equation*}
		where $\alpha,\alpha'$ are multi-indices, and $\xi^{\alpha}=(\xi_1^{\alpha_1},...,\xi_n^{\alpha_n})$ and the derivatives are understood after taking normal coordinates around the point $x$.
	\end{prop}
	
	The proof of Proposition \ref{two points function short} follows directly from the following two lemmas. In the case that the geodesic flow is periodic, we have a full asymptotic expansion for the spectral projector kernel \cite{Z97}, see also \cite{Z2009}. In particular, we have the following:
	\begin{lem}[Zelditch]
		\label{periodic}
		Let $(M,g)$ be a compact, real-analytic manifold with empty boundary. Suppose that $M$ is Zoll, then the conclusion of Proposition \ref{two points function short} holds.
	\end{lem}
	The second lemma is borrowed from Canzani-Hanin \cite{CH15,CH18}, see also the preceding work of Safarov \cite{Sa88}:
	\begin{lem}
		\label{aperiodic}
		Let $(M,g)$ be a compact, real-analytic manifold with empty boundary. Suppose that $x\in M$ is not self-focal, as in
		Definition \ref{focal and twisted}, then the conclusion of Proposition \ref{two points function short} holds.
	\end{lem}

	We are finally ready to prove Proposition \ref{covariance function}:
	\begin{proof}[Proof of Proposition \ref{covariance function}]
		First we observe that, thanks to Lemma \ref{lem: zelditch}, under the assumptions of Proposition \ref{covariance function}, the conclusion of Proposition \ref{two points function short} holds for almost all $x\in M$. Indeed,	thanks to Lemma \ref{lem: aper or per}, we may assume that the geodesic flow on $M$ is either aperiodic or periodic. In the latter case, the conclusion of Proposition \ref{two points function short} holds for all $x\in M$. In the former case, Lemma \ref{aperiodic} holds for almost all $x\in M$.
		
		Hence, we are left with showing how the conclusion of Proposition \ref{two points function short} implies the conclusion of Proposition \ref{covariance function}. Since $\rho(T)=o(T)$, re-writing the integral in \eqref{2.2} the spherical coordinates,  we have
		\begin{align}
			\sum_{\lambda_i\in[T- \rho(T), T]} \phi_{i}(y')\phi_{i}(y)&=  (2\pi)^{-n} c_M\rho(T)T^{n-1} \int_{|\xi|=1} e(\langle Td_g(y',y), \xi\rangle)d\xi \nonumber \\   &+ O\left(\rho(T)^2T^{n-1}d_g(y',y)\right) +o(T^{n-1}) \nonumber \\
			&= c_M \rho(T)T^{n-1}\frac{J_{\Lambda}(|T d_g(x,y)|)}{|T d_g(y',y)|^{\Lambda}}  +O\left( \rho(T)^2T^{n-1}d_g(y',y)\right) + o(T^{n-1}),\label{two points functions short2}
		\end{align}
		where $\Lambda= (n-2)/2$, and Proposition \ref{covariance function} follows.
	\end{proof}
	\subsection{Sogge's bound}
	\label{sogge}
	Let $\phi_{i}$ be an eigenfunction with eigenvalue $\lambda_i^2$. Then, we have the following estimate on the $L^p$ norms of $\phi_{i}$ \cite{S88}, see also \cite[Theorem 10.1]{Zbook}:
	\begin{align}
		||\phi_{i}||_{L^p}\ll \lambda_i ^{\sigma(p)} ||\phi_{i}||_{L^2}\label{Sogge's bound}
	\end{align}
	where
	\begin{align}
		\sigma(p)=\begin{cases}
			\frac{n-1}{2}\left(\frac{1}{2}-\frac{1}{p}\right) & 2< p \leq \frac{2(n+1)}{n-1} \\
			n\left(\frac{1}{2}-\frac{1}{p}\right)-1/2 & p\geq \frac{2(n+1)}{n-1}.
		\end{cases} \nonumber
	\end{align}
	
	\section{Asymptotic Gaussianity}
	\label{asymptotic gaussianity}

	Before stating the main result of this section, we need to introduce some notation.  We denote $F_{\mu}$ to be the monochromatic isotropic Gaussian field on $B_0\subset \mathbb{R}^n$ with spectral measure $\mu$, the Lebesgue measure on the $n-1$ dimensional sphere $S^{n-1}$. Equivalently, $F_{\mu}$ has the covariance function
	$$\mathbb{E}[ F_{\mu}(y)\cdot F_{\mu}(y')]= \int_{|\xi|=1}e\left(\langle y-y', \xi\rangle\right)d\xi=\frac{J_{\Lambda}(|y-y'|)}{|y-y'|^{\Lambda}},$$
	where $\Lambda=(n-2)/2$. In what follows we will use the shorthands
	\begin{align}
		&\mathcal{V}(F_x):=\mathcal{V}\left(F_x, \frac{1}{2}B_0\right) \quad\quad\quad\text{and} \quad\quad\quad \mathcal{V}(F_\mu):= \mathcal{V}\left(F_{\mu}, \frac{1}{2}B_0\right) \nonumber.
	\end{align}
	The aim of this section is to prove the following proposition:
	\begin{prop}
		\label{asymptotic Gaussianity}
		Let $F_x$ be as in \eqref{def of F_x} and $F_{\mu}$ be as above. Under the assumptions of Theorem \ref{main thm}, there exists a subset $A\subset M$ of volume at most $O(\log T/T)$ such that, uniformly for all $x\in M \backslash A$, we have
		\begin{align}
			&\mathcal{V}(F_x) \overset{d}{\longrightarrow} \mathcal{V}(F_{\mu}) &T\rightarrow \infty, \nonumber\end{align}
		the convergence in distribution.
	\end{prop}
	To ease the exposition, we split the proof of Proposition \ref{asymptotic Gaussianity} into a series of steps: first we prove that $F_{x}$, as in \eqref{def of F_x}, converges to $F_{\mu}$ in distribution in the appropriate space of functions, then we deduce Proposition \ref{asymptotic Gaussianity} using a stability notion for the nodal set introduced in \cite{NS}. Before beginning the proof, we need the following consequence of Sogge's bound \eqref{Sogge's bound}.
	\subsection{Consequence of Sogge's bound}
	In this section we prove the following consequence of \eqref{Sogge's bound}:
	\begin{lem}
		\label{sup bound}
		Let $K=K\geq 1$ be some parameter which may depend on $T$, and $v(T)$ be as in \eqref{defv}.
		Then there exists a subset $A\subset M$ of volume at most $O(K^{2\frac{n+1}{n-1}}T^{-1})$ such that
		\begin{enumerate}
			\item uniformly for all $x\in M\backslash A$ we have
			$$		\max_{\lambda_i\in[T-\rho(T), T]}  \sup_{B_g(x,2/T)}|\phi_i| \lesssim K^{-1}v(T)^{1/2} $$
			\item uniformly for all $x\in M\backslash A$ we have
			$$		\max_{\lambda_i\in[T-\rho(T), T]}  \sup_{B_g(x,2/T)}|T^{-1}\nabla\phi_i| \lesssim K^{-1}v(T)^{1/2} $$
		\end{enumerate}
	\end{lem}

	To state the next result, given a Laplace eigenfunction $\phi_{i}$, we denote by $\phi_{i,x}$ the  restriction of $\phi_{i}$ to $B_g(x,4/T)$ via the exponential map, that is
	$$\phi_{i,x}(y)= \phi_i(\exp_x(y/T)),$$
	for $y\in B(0,4)$ (here we tacitly assume that $T$ is sufficiently large so that $4/T$ is less than the injectivity radius), see also the definition of $F_x$ at the beginning of section \ref{local Weyl's law}. With this notation in mind, we prove the following, see also \cite{ST94}:
	\begin{lem}
		\label{elliptic comparison}
		Let $T\geq 1$ and  $\phi_{i}$  let be a Laplace eigenfunction with $\lambda_i \in[ T-\rho(T),T]$, then
		\begin{enumerate}
			\item for all $x\in M$, we have
			$$ \sup_{B_g(x,2/T)}|\phi_i|^2 \lesssim \int_{B(0,4)} |\phi_{i,x}(y)|^2 dy.$$
			\item for all $x\in M$, we have
			$$\sup_{B_g(x,2/T)}|T^{-1}\nabla\phi_i|^2 \lesssim  \int_{B(0,4)} |\phi_{i,x}(y)|^2 dy.$$
		\end{enumerate}
	\end{lem}
	\begin{proof}
		Given $\phi_{i}$, let us consider the function $h(x,t)= \phi_{i}(x)e^{\lambda_i t}$ defined on $M\times [-2,2]$ and let us write $h_{T}(\cdot)= h (T^{-1}\cdot)$. Then, since the supremum norm is scale invariant, we have
		\begin{align}
			\label{comparioson 1}
			\sup_{B_g(x,2/T)}|\phi_i|\lesssim \sup_{B_g(x,2/T)\times [-2/T,2/T]}|h| \lesssim ||h_T||_{L^{\infty}(\mathcal{B})}
		\end{align}
		\begin{align}
			\label{comparison 2} \sup_{B_g(x,2/T)}|\nabla\phi_i| \lesssim \sup_{B_g(x,2/T)\times [-2/T,2/T]}|\nabla h| \lesssim  T ||h_T||_{C^1(\mathcal{B})},
		\end{align}
		where $\mathcal{B}= B_g(x,2)\times [-2,2]$.   Since $h$ is an harmonic function ($\Delta h=0$), for any $k\geq 0$, elliptic regularity \cite[Page 330]{Ebook}, gives
		\begin{align}
			\label{3.1.1} ||h_T||_{H^k(\mathcal{B})}\lesssim_k ||h_T||_{L^2(2\mathcal{B})}
		\end{align}
		where $H^k$ is the Sobolev's norm. Bounding the $C^1$ norm by the $H^k$ norm for $k$ sufficiently large depending on the dimension of $M$, the lemma follows by inserting \eqref{3.1.1} into \eqref{comparioson 1} and \eqref{comparison 2}, and noticing that $||h_T||_{L^2(2\mathcal{B})} \lesssim || \phi_{i,x}||_{L^2(B(0,4))}$.
	\end{proof}

	\begin{proof}[Proof of Lemma \ref{sup bound}]
		First, we observe that, given $p\geq 2$, the function $x\rightarrow x^{p/2}$ is convex for $x\geq 0$. Therefore, applying Jensen's inequality to part (1) of Lemma \ref{elliptic comparison}, we obtain
		\begin{align}
			\left(	\sup_{B_g(x,2/T)} |\phi_{i}|\right)^p \lesssim_p  \left(  \int_{B(0,4)} |\phi_{i,x}(y)|^2dy\right)^{p/2} \lesssim_p\int_{B(0,4)} |\phi_{i,x}(y)|^p dy \label{L^p1}
		\end{align}
		and, similarly
		\begin{align}
			\label{L^p2}
			\left(\sup_{B_g(x,2/T)}|T^{-1}\nabla\phi_i|\right)^p  \lesssim_p \int_{B(0,4)} |\phi_{i,x}(y)|^p dy.
		\end{align}
		
		We are now going to prove part (1) of Lemma \ref{sup bound}. By Sogge's bound \eqref{Sogge's bound} with $p\leq 2(n+1)/(n-1)$, bearing in mind that $||\phi_i||_{L^2}=1$, we have
		\begin{align}
			\left(\int_{M} |\phi_{i}(x)|^pdx \right)^{1/p}\lesssim T^{\frac{n-1}{2}\left(\frac{1}{2}- \frac{1}{p}\right)}=: \tilde{T} \nonumber
		\end{align}
		for all $\lambda_i \leq T$. Thus, integrating both sides of \eqref{L^p1} with respect to $x\in M$ and exchanging the order of the integrals, we obtain
		$$ \int_M\left(	\sup_{B_g(x,2/T)} |\phi_{i}|\right)^p dx \lesssim \int_{B(0,4)} \int_M|\phi_{i,x}(y)|^p dx dy\lesssim \tilde{T}^p.$$
		Therefore, by Chebyshev's bound, for any $K_1>0$, we have
		\begin{align}
			\vol_g\left( \left\{x\in M: \sup_{B_g(x,2/T)}|\phi_{i}|\geq K_1\right\} \right) \lesssim K_1^{-p} \tilde{T}^p,\nonumber
		\end{align}
		and,	taking the union bound over the $O(v(T))$ choices for $i$, we deduce
		\begin{align} \label{3.1.3}	\vol_g\left( \left\{x\in M: \max_{\lambda_i \in [T-\rho(T),T]}\sup_{B_g(x,2/T)}|\phi_{i}|\geq K_1\right\} \right) \lesssim K_1^{-p} v(T) \cdot \tilde{T}^p. \end{align}
		Thus,	taking $K_1=K^{-1}v(T)^{1/2}\gtrsim K^{-1}(\rho(T) T^{n-1})^{1/2}$ in \eqref{3.1.3} and recalling that $\tilde{T}=T^{\frac{n-1}{2}\left(\frac{1}{2}- \frac{1}{p}\right)}$, we have
		\begin{align}\vol_g\left( \left\{x\in M: \max_{\lambda_i \in [T-\rho(T),T]}\sup_{B_g(x,2/T)}|\phi_{i}|\geq K^{-1} v(T)^{1/2}\right\}  \right)\lesssim K^p \rho(T)^{-\frac{p}{2}+1} T^{\nu(n,p)}, \label{bound volume}
		\end{align}
		where
		\begin{align} \nu(n,p):&= -p\frac{n-1}{2}  + n-1 + \frac{n-1}{2}\left(\frac{p}{2}-1\right)  \nonumber \\
			&= \frac{n-1}{2}\left(1- \frac{p}{2}\right) \nonumber
		\end{align}
		Hence, taking $p=2(n+1)/(n-1)$ in \eqref{bound volume}, we have
		$$\vol_g\left( \left\{x\in M: \max_{\lambda_i \in [T-\rho(T),T]}\sup_{B_g(x,2/T)}|\phi_{i}|\geq K^{-1} v(T)^{1/2}\right\}  \right)\lesssim K^p  T^{-1},$$
		as required. Thanks to \eqref{L^p2}, the proof of part (2) of Lemma \ref{sup bound} follows step by step the proof of part (1).
	\end{proof}
	\subsection{Convergence to Gaussian}
	\label{convergencetoGaussian}
	In this section, we prove the following:
	\begin{lem}[Convergence of finite dimensional distributions]
		\label{convfindimdis}
		Let $m$ be some positive integer and recall that $B_0=B(0,1)$, moreover let $F_x$ and $F_{\mu}$ be as in section \ref{asymptotic gaussianity}. Under the assumptions of Theorem \ref{main thm}, there exists a set $A\subset M$ of volume at most $O(\log T/T)$ such that the following holds: pick $m$ points  $y_1,...y_m\in B_0\subset \mathbb{R}^n$,   then, uniformly for $x\in M\backslash A$, we have
		\begin{align}&(F_x(y_1),..., F_x(y_m))\overset{d}{\longrightarrow}(F_{\mu}(y_1),..., F_{\mu}(y_m)) &T\rightarrow \infty. \nonumber
		\end{align}
		and for any $\alpha=(\alpha_1,...,\alpha_n)$, where $n$ is the dimension of $M$, with $|\alpha|\leq 2$, we have
		\begin{align}& (D^{\alpha} F_x(y_1),..., D^{\alpha}F_x(y_m))\overset{d}{\longrightarrow}( D^{\alpha}F_{\mu}(y_1),..., D^{\alpha}F_{\mu}(y_m)) &T\rightarrow \infty. \nonumber
		\end{align}
	\end{lem}
	\begin{proof}[Proof of Lemma \ref{convfindimdis}]
		Let $\phi_{i,x}$ the  restriction of $\phi_{i}$ to $B_g(x,1/T)$ and let $A$ be the set given in Lemma \ref{sup bound} applied with $K= (\log T)^{\frac{n-1}{2(n+1)}}=(\log T)^{c}$ together with the set of exceptional points given in Proposition \ref{covariance function}. By  Lemma \ref{sup bound}, for all $x\in M\backslash A$, we have
		\begin{align}
			\max_{\lambda_i\in[T-\rho(T), T]} \sup_{B_g(x,2/T)} |\phi_i| \lesssim\frac{ v(T)^{1/2}}{(\log T)^c}
			\label{supnormsmall}
		\end{align}
		and, since $\sup_{B_0} |\nabla \phi_{i,x}|\lesssim \sup_{B_g(x,2/T)}|T^{-1}\nabla \phi_{i}|$, we also have
		\begin{align}
			\max_{\lambda_i\in[T-\rho(T), T]} \sup_{B_0}  |\nabla \phi_{i,x}| \lesssim\frac{ v(T)^{1/2}}{(\log T)^c}.
			\label{supnormsmal2}
		\end{align}
		To simplify the exposition, from now on, we assume that $x\in M\backslash A$.
		
		In order to prove the first claim of the lemma, by the multidimensional version of Lindeberg's Central Limit Theorem and in light of Proposition \ref{covariance function}, it is sufficient to prove that for any $\varepsilon>0$ we have
		\begin{align} \label{2}
			&\frac{1}{v(T)^2}\sum_{\lambda_i} \mathbb{E}[|a_{i}\phi_{i,x}(y)|^2\mathds{1}_{|a_{i}\phi_{i,x}(y)|>\varepsilon v(T)}]\rightarrow 0 & T\rightarrow \infty
		\end{align}
		uniformly for all $y\in B_0$, where $\mathds{1}$ is the indicator function and $v(T)^2= c_M\rho(T)T^{n-1}(1+o(1))$. Thanks to Lemma \ref{diagonal Weyl}, and since $\phi_{i,x}$ are deterministic, we have
		\begin{align} \frac{1}{v(T)^2}\sum_{\lambda_i} \mathbb{E}[|a_{i}\phi_{i,x}(y)|^2\mathds{1}_{|a_{i}\phi_{i,x}(y)|>\varepsilon v(T)}]&= \frac{1}{v(T)^2}\sum_{\lambda_{i}}|\phi_{i,x}(y)|^2 \mathbb{E}[|a_{i}|^2\mathds{1}_{|a_{i}\phi_{i,x}(y)|>\varepsilon v(T)}] \nonumber \\
			&	\leq \sup_{i}\mathbb{E}[|a_{i}|^2\mathds{1}_{|a_{i}\phi_{i,x}(y)|>\varepsilon v(T)}], \nonumber
		\end{align}
		thus, in order to prove \eqref{2}, it is enough to show that
		\begin{align}\label{2.1}
			&\sup_{i}\mathbb{E}[|a_{i}|^2\mathds{1}_{|a_{i}\phi_{i,x}(y)|>\varepsilon v(T)}]\rightarrow 0 & T\rightarrow \infty.
		\end{align}
		Thanks to \eqref{supnormsmall}, we have  \begin{align}\mathds{1}_{|a_{i}\phi_{i,x}(y)|>\varepsilon v(T)}\leq \mathds{1}_{|a_{i}|\gtrsim\varepsilon (\log T)^c}, \nonumber
		\end{align}
		thus
		$$ \lim\limits_{T\rightarrow \infty}\mathbb{E}[|a_{i}|^2\mathds{1}_{|a_{i}\phi_{i,x}(y)|>\varepsilon v(T)}] \leq  \lim\limits_{M\rightarrow \infty}\lim\limits_{T \rightarrow \infty} \int_{\varepsilon (\log T)^{c}}^M t^2d\mathbb{P}(|a_{i}|>t)=0 $$
		where we have swapped limits using Fubini's Theorem and the fact that $\mathbb{E}[|a_{i}|^2]=1$. This concludes the proof of \eqref{2.1}.

		In order to prove the second claim of the lemma, and upon recalling the second part of Proposition \ref{covariance function}, again by the multidimensional version of Lindeberg's Central Limit Theorem, it is enough to prove that
		for any $\varepsilon>0$ and $|\alpha|\leq 2$ we have
		\begin{align} \label{2der}
			&\frac{1}{v(T)^2}\sum_{\lambda_i} \mathbb{E}[|a_{i}D^{\alpha}\phi_{i,x}(y)|^2\mathds{1}_{|a_{i}D^{\alpha}\phi_{i,x}(y)|>\varepsilon v(T)}]\rightarrow 0 & T\rightarrow \infty
		\end{align}
		uniformly for all $y\in B_0$.	Similarly to the above argument, \eqref{supnormsmal2} implies \eqref{2der} if $|\alpha|=1$; for $|\alpha|=2$ we use the Helmholtz's equation to bound the second derivatives.
	\end{proof}
	
	\subsection{Tightness}
	A  sequence of probability measures $v_n$ on some topological space $X$ is \textit{tight} if for every $\epsilon>0$, there exists a compact set $K=K(\epsilon)\subset X$ such that
	$$\nu_n(X\backslash K)\leq \epsilon,$$
	uniformly for all $n\geq 0$. We will need the following lemma, borrowed from \cite[Lemma 1]{Pri93}, see also \cite[Chapter 6 and 7]{BI}, which characterise tightness  in the space of continuously twice differentiable functions:
	\begin{lem}[Tightness]
		\label{tightness}
		Let $V$ be a compact subset of $\mathbb{R}^n$, let $\nu_n$ be a sequence of probability measures on $C^2(V)$, continuously  twice differentiable functions on $V$, then $\nu_n$ is tight if the following conditions hold:
		\begin{enumerate}
			\item For any $|\alpha|\leq 2$, there exists some $y\in V$ such that for every $\varepsilon>0$  there exists $M>0$  with
			\begin{align}
				\nu_n(g\in C^2(V): |D^{\alpha}g(y)|>M)\leq \varepsilon. \nonumber
			\end{align}
			
			\item For any $|\alpha|\leq 2$ and $\varepsilon>0$, we have
			$$\lim_{\delta\rightarrow 0}\limsup_{n\rightarrow \infty}\nu_n\left(g\in C^2(V): \sup_{|y-y'|\leq \delta}|D^{\alpha}g(y)- D^{\alpha}g(y')|> \varepsilon\right)=0.$$
		\end{enumerate}
	\end{lem}
	We wish to apply Lemma \ref{tightness} to $V=\overline{B_0}$ with $\nu_{T,x}$ being the sequence of probability measures on $C^2(V)$ induced by the pushforward measure of $F_x$. That is, for an open set $F\subset C^2(V)$, we let
	\begin{align}
		\label{3.1}\nu_{T,x}(F):= (F_x)_{*}\mathbb{P}(A)= \mathbb{P}(F_x(\omega,\cdot)\in F).\end{align}
	We then have the following:
	\begin{lem}
		\label{nu_T tight}
		Let $V= \overline{B_0}$ and let $\nu_{T,x}$ be as in \eqref{3.1}, then, under the assumptions of Theorem \ref{main thm}, for almost all $x\in M$ the sequence $\nu_T$ is tight.
	\end{lem}
	\begin{proof}
		For brevity let us write $\nu_T= \nu_{T,x}$. For condition $(1)$ of Lemma \ref{tightness}, we observe that Proposition \ref{covariance function} implies that $$\mathbb{E}\left[|D^{\alpha}F_x(0)|^2\right]\lesssim 1,$$
		for $|\alpha|\leq 2$. Thus, Chebyshev's inequality give
		$$\mathbb{P}\left(|D^{\alpha}F_x(0)| >M \right)\lesssim M^2,$$
		and condition (1) follows by taking $M= \epsilon^{-1/2}$.

		For	condition $(2)$ of Lemma \ref{tightness}, since $F_x$ is almost surely analytic, we have
		\begin{align}\sup_{|y-y'|\leq \delta} |D^{\alpha}F_x(y)- D^{\alpha}F_x(y')|\lesssim \sup_{B_0}|\nabla D^{\alpha}F_x|\delta. \label{bound continuity}
		\end{align}
		Therefore it is sufficient to prove the following claim:
		\begin{align}
			\label{claim contonuity}
			\mathbb{P}(	\sup_{B_0}|\nabla D^{\alpha}F_x|>M)\lesssim M^{-2}.
		\end{align}
		Indeed, \eqref{claim contonuity} together with \eqref{bound continuity} imply condition (2) by choosing $M=\epsilon \delta ^{-1}$.
		
		We are now going to prove \eqref{claim contonuity}. First we observe that, for any fixed $k\geq 0$, Proposition \ref{covariance function}  gives
		\begin{align}
			\label{3.2.2} \mathbb{E}[|D^{\alpha}F_x(y)|^2]\lesssim 1
		\end{align}
		uniformly for all $y\in B_0$ and $|\alpha|\leq k$, where the constant implied in the $\lesssim$-notation is absolute. Now Sobolev's inequality gives that
		$$ || F_x||_{C^2(B_0)} \lesssim ||F_x||_{H^t(B_0)}$$
		for some $t$ sufficiently large depending on $n$, and the constant implied in the $\lesssim$-notation is independent of $T$. Therefore, by \eqref{3.2.2}, we have
		\begin{align}
			\mathbb{E} [ || F_x||_{C^2(B_0)}] \lesssim   \mathbb{E}[||F_x||_{H^t(B_0)}]= O(1) \label{3.2.3}
		\end{align}
		where the constants implied in the $\lesssim$-notation is independent of $T$. The inequality \eqref{3.2.3} together with Chebyshev's inequality implies \eqref{claim contonuity}. This concludes the proof of \eqref{claim contonuity} and thus of Lemma \ref{nu_T tight}.
	\end{proof}
	Since $v_T$ is tight, the convergence of finite-dimensional distributions implies weak$^{\star}$ convergence. That is, combining Lemma \ref{convfindimdis} and Lemma \ref{tightness}, we proved the following lemma, see for example \cite[Section 7]{BI13}:
	\begin{lem}
		\label{distribution F}
		Let $V= \overline{B_0}$ and let $\nu_{\infty}$ the pushforward of $F_{\mu}$ on $C^2(V)$, where $F_{\mu}$ is as in
		section \ref{asymptotic gaussianity}. Under the assumptions of Theorem \ref{main thm}, there exists a set $A\subset M$ of volume at most $O(\log T/T)$ such that, uniformly for all $x\in M \backslash A$,  $\nu_{T,x}$ weak$^{\star}$ converges to $\nu_{\infty}$ in the space of probability measures on $C^2(V)$.
	\end{lem}
	\subsection{Concluding the proof of Proposition \ref{asymptotic Gaussianity}} To conclude the proof of Proposition \ref{asymptotic Gaussianity}, we just need the following Lemma, see for example \cite[Lemma 6.2]{RS20}, which shows that $\mathcal{V}(\cdot)$, that is the nodal volume, is a continuous map on the appropriate space of functions:
	\begin{lem}
		\label{continouty}
		Let $B\subset\mathbb{R}^n$ be a ball, define the set $C^2_{*}(\overline{2B})=\{g\in C^2(\overline{2B}):  |g|+|\nabla g|>0 \}$. Then $\mathcal{V}(\cdot, B)$ is a continuous functional on $C^2_{*}(2B)$.
	\end{lem}
	We are now in a position to prove Proposition \ref{asymptotic Gaussianity}.
	\begin{proof}[Proof of Proposition \ref{asymptotic Gaussianity}]
		Let $A$ be given by Lemma \ref{distribution F}. Let $\overline{B_0}=V$,	by Bulinskaya's lemma, see for example \cite{So}, $F_{\mu}\in C^2_{*}(V)$ almost surely. Thus, Lemma \ref{distribution F} and the Continuous Mapping Theorem imply that
		\begin{align} \nonumber &\mathcal{V}(F_x)\overset{d}{\rightarrow} \mathcal{V}(F_{\mu}) &T\rightarrow \infty,
		\end{align}
		as required.
	\end{proof}
	\begin{rem}
		If $\rho(T)\rightarrow \infty$, the conclusion of Proposition \ref{asymptotic Gaussianity} holds without the need of removing an exceptional set and the proof is considerably simpler. Indeed, for any eigenfunction $\phi_{i}$, we have \cite[page 39 and page 105]{Zbook},
		\begin{align}
			&||\phi_{i}||_{L^{\infty}}\leq \lambda_i^{\frac{n-1}{2}} & ||\nabla\phi_{i}||_{L^{\infty}}\leq \lambda_i ||\phi_{i}||_{L^{\infty}}. \label{r.1}
		\end{align}
		Thus, if $\rho(T)\rightarrow \infty$, we have
		\begin{align}
			\label{sup bounds}
			& ||\phi_{i}||_{L^{\infty}}/ v(T)^{1/2} \rightarrow 0 & T\rightarrow \infty \nonumber \\
			& ||\nabla\phi_{i,x}||_{L^{\infty}}/ v(T)^{1/2} \rightarrow 0 & T\rightarrow \infty
		\end{align}
		where $v(T)$ is as in \eqref{defv} and $\phi_{i,x}$ is as in section \ref{sogge}. From \eqref{sup bounds}, we obtain Lemma \ref{convfindimdis} without the need to remove any exceptional set.
		
		On the other hand, the extreme case $\rho(T)=1$ includes the $n$-dimensional sphere, where \eqref{r.1} are known to be sharp. Thus, the exceptional set cannot be removed entirely.
	\end{rem}
	
	\section{$\mathcal{V}(F_x)$ and the doubling index}
	\label{doubling and harmonic lift}
	The aim of this section is to show that we can bound $\mathcal{V}(F_x)$ using the \textit{doubling index} of $F_x$ on $B_0$; in doing so,  we follow the work Jerison and Lebeau \cite[Section 14]{DIbook} which, in turn, is based on \cite{Lin91}, see also Kukavica \cite{Ku952,Ku95,Ku98} for a different approach. We first define the \textit{doubling index}, following \cite{LMlecturenotes,LM2} and \cite{DF}: given a (geodesic) ball $B\subset M$ and a function $g:3B\rightarrow \mathbb{R}$, we define the \textit{doubling index} of $g$ on $B$, denoted by $N(g,B)$, as
	\begin{align}\label{doubling index def}
		N(g,B)=\log \frac{\sup_{2B} |g|}{\sup_B |g|}.
	\end{align}
	It will be useful to also introduce an auxiliary function $f^H$ sometimes known as the \textit{harmonic lift} of $f$.  Following \cite[Page 231]{DIbook} (see also \cite{LM2}), we define $	f^H: M\times \mathbb{R} \rightarrow \mathbb{R}$ to be the unique solution of
	\begin{align} \label{harmonification}
		(\Delta+ \partial_t^2)f^H(x,t)=0 \hspace{10mm} f^H(x,0)=0 \hspace{10mm}\partial_t f^H(x,0)= f.
	\end{align}
	Note that we can write explicitly $f^H$ as
	\begin{align}
		f^H(x,t)=v^{-1/2}(T)\sum_{\lambda_i \in [T-\rho(T),T]} a_{i}\frac{\sinh(\lambda_i t)}{\lambda_i}\phi_{i}(x) \label{f^h},
	\end{align}
	where $a_{i}$ and $\phi_{i}$ are as in \eqref{function} and $v(T)$ is as in \eqref{defv}. Given a \textquotedblleft ball\textquotedblright, of radius $r>0$, $B\subset M\times \mathbb{R}$, which is defined to be $B=\tilde{B}\times I$ for a geodesic ball $\tilde{B}\subset M $ of radius $r>0$ and some interval $I\subset \mathbb{R}$ of length $r>0$, we denote by $B^{+}:=B \cap (M \times[0,\infty))$ and, given $r>0$, $rB^{+}:=rB \cap (M \times[0,\infty))$. We also define the doubling index of a function $g: 3B\subset M\times \mathbb{R}$ on $B$ or $B^{+}$ analogously to \eqref{doubling index def}. With this notation in mind, in this section we prove the following result:
	\begin{prop}
		\label{doubling index bound}
		Let $f$ and  $f^H$ be as in \eqref{function} and  \eqref{f^h} respectively. Then there exists some $\eta=\eta(M)>0$ such that the following holds: for every ball $B=B(x,r) \subset M\times \mathbb{R}$ centred at a point lying on $x\in M \cong M\times \{0\}$ of radius $r<\eta/10$, we have
		$$\mathcal{V}\left(f, \frac{1}{2}B \cap M\right)\lesssim N(f^H, 2B^{+}) .$$
	\end{prop}
	
	Our main tool to prove Proposition \ref{doubling index bound} is the following general result which we borrow from \cite[Theorem 14.7]{DIbook} and \cite[Proposition 6.7]{DF}:
	\begin{lem}
		\label{doubling index}
		Let $B^{\mathbb{C}}\subset \mathbb{C}^n$ be a ball of radius $1$, and let $H$ be a holomorphic function on $2B^{\mathbb{C}}$. If
		\begin{align}
			|H|_{L^{\infty}(2B^{\mathbb{C}})}\leq e^{CN} |H|_{L^{\infty}(B^{\mathbb{C}})} \nonumber
		\end{align}
		for some $C>0$, then
		\begin{align}
			\mathcal{H}^{n-1}\left(\{H=0\} \cap \frac{1}{2}B \cap \mathbb{R}^n\right)\lesssim N . \nonumber
		\end{align}
	\end{lem}
	
	To ease the exposition, we split the proof of Proposition \ref{doubling index bound} into a series of steps: using the fact that $M$ is real analytic we will bound $\mathcal{V}(f,\cdot)$ by the order of growth of the complex extension of $f$, defined in section \ref{Grauert tube} below. We will then show how $N(f^H,\cdot)$ controls the said order of growth.

	\subsection{Complexification of $f$}
	\label{Grauert tube}

	Since $(M, g)$ is real analytic and compact, by the Bruhat-Whitney Theorem \cite{WB59} there exists a complex manifold $M^{\mathbb{C}}$ where $M$ embeds as a totally real manifold. Moreover, it is possible  to analytically continue any Laplace eigenfunction $\phi_{i}$ to a homomorphic function $\phi_{i}^{\mathbb{C}}$ defined on  a maximal uniform \textit{ Grauert tube}, that is there exists some $\eta>0$ such that $\phi_{i}^{\mathbb{C}}$ is an homomorphic function on
	\begin{align}
		M_{\eta}^{\mathbb{C}}:=\{\zeta \in M^{\mathbb{C}}:	\sqrt{\gamma}(\zeta) < \eta \} \label{grauert tube}
	\end{align}
	where $\sqrt{\gamma}(\cdot)$ is the Grauert tube function, for details see \cite[Chapter 14]{Zbook}. For notational brevity, from now on we will write $M^{\mathbb{C}}=M_{\eta}^{\mathbb{C}}$, as the precise value of $\eta$ will be unimportant, and let $f^{\mathbb{C}}$, defined on $M^{\mathbb{C}}$, be the complexification of $f$. The next Lemma quantify the growth of $f^{\mathbb{C}}$ in terms of the doubling index of $f^H$.
	\begin{lem}
		\label{complexification}
		Let $f$ be as in \eqref{function}, $f^H$ be  as in \eqref{harmonification} and let $f^{\mathbb{C}}$ be the complexification of $f$. Let  $B\subset  M \times \mathbb{R}$ be a  ball  centred at a point lying on $M\cong M \times \{0\}$  of radius less than $\eta/3$, where $\eta$ is as in \eqref{grauert tube}. Suppose that, for some $N>0$, one has
		\begin{align} \nonumber
			||f^H||_{L^{\infty}({3B^{+}})}\lesssim e^N ||f^H||_{L^{\infty}(B^{+})}
		\end{align}
		then
		\begin{align} \nonumber
			||f^{\mathbb{C}}||_{L^{\infty}(2(B\cap M)^{\mathbb{C}})}\leq C' e^{CN}	||f||_{L^{\infty}(B\cap M)}.
		\end{align}
		for some absolute constant $C>0$ and $C'=C'(\eta)>0$.
	\end{lem}
	To prove Lemma \ref{complexification}, we first need the following result borrowed from \cite[Page 231]{DIbook}.
	\begin{lem}
		\label{estimate DIbook}
		Let $B=B(x,r)\subset M \times \mathbb{R} $ be a ball centred at a point lying on $M\cong M\times \{0\}$ of radius $r>0$. Then, uniformly for all $B$, there exists some $0< \beta<1$ such that
		\begin{align}
			||f^H||_{L^{\infty}(B^{+})}\lesssim_{\beta,M} || \partial_t f^H||^{\beta}_{L^{\infty}(B\cap M)}|| f^H||^{1-\beta}_{L^{\infty}(2B^{+})}. \nonumber
		\end{align}
	\end{lem}
	We are now ready to prove Lemma \ref{complexification}
	\begin{proof}[Proof of Lemma \ref{complexification}]
		Since $f^H$ satisfies
		$$(\partial_t^2 +\Delta)f^H=0,$$
		the hypoellipticity of the operator $\partial_t^2 +\Delta$, see for example \cite[Lemma 7.5.1 and equation (4.4.1)]{H63}, for any multi-index $|\alpha|>0$, gives
		\begin{align} \label{6}
			\sup_{B^{+}}|D ^{\alpha}f^H|\lesssim C^{|\alpha|}	|| f^H||_{H^1\left(\frac{5}{2}B^{+}\right)}
		\end{align}
		for some $C>1$, where $H^1$ is the Sobolev norm. By elliptic regularity (we can compare the $H_1$ norm of $f^H$ with the $L^{\infty}$ norm of $f^H$, see \cite[Page 330]{Ebook}), we also have
		\begin{align}\label{7}
			|| f^H||_{H^1\left(\frac{5}{2}B^{+}\right)} \lesssim ||f^H||_{L^{\infty}(3B^{+})}.
		\end{align}
		Moreover, the assumption on the doubling index of $f^H$ and Lemma \ref{estimate DIbook} give
		\begin{align} \label{8.1}
			||f^H||_{L^{\infty}(3B^{+})}\lesssim e^N ||f^H||_{L^{\infty}(B^{+})}\lesssim e^{CN} ||f||_{L^{\infty}(B \cap M)}
		\end{align}
		for some $C>0$. Therefore, putting \eqref{6}, \eqref{7} and \eqref{8.1} together, we obtain
		\begin{align} \label{8}
			\sup_{2(B \cap M)}|D ^{\alpha}f|\lesssim e^{CN +C'|\alpha|}\sup_{B \cap M} |f|.
		\end{align}
		Since $f$ is real analytic, we can write
		$$f^{\mathbb{C}}(z)= \sum_{\alpha} \frac{D^{\alpha} f(x)}{|\alpha| !} z^{|\alpha|}$$
		thus, Lemma \ref{complexification} follows by \eqref{8}.
	\end{proof}
	
	\subsection{Concluding the proof of Proposition \ref{doubling index bound} }
	
	\begin{proof}
		First, Lemma \ref{doubling index} gives
		\begin{align}
			\mathcal{V}\left(f, \frac{1}{2}B \cap M \right)\lesssim \log \frac{\sup_{2(B \cap M)^{\mathbb{C}}}|f^{\mathbb{C}}|}{\sup_{(B \cap M)^{\mathbb{C}}}|f^{\mathbb{C}}|} \label{4f1}
		\end{align}
		where $f^{\mathbb{C}}$ is the complexification of $f$, note that here we have used the assumption that $r<\eta/10$ where $\eta$ is given as in section \ref{grauert tube}. Now let $N= N(f^H,2B^{+})$, then Lemma \ref{complexification}  imply that
		\begin{align}
			\log \frac{\sup_{2(B \cap M)^{\mathbb{C}}}|f^{\mathbb{C}}|}{\sup_{(B \cap M)^{\mathbb{C}}}|f^{\mathbb{C}}|} \lesssim N  \label{4f2}.
		\end{align}
		Hence, the Proposition follows by combining \eqref{4f1} and \eqref{4f2}.
	\end{proof}
	
	\subsection{Estimates for the doubling index}
	In this section, we collect various estimates for $N(f^H,\cdot)$ encountered in section \ref{doubling and harmonic lift}. In light of Proposition \ref{doubling index bound} the said estimates will directly produce bounds for $\mathcal{V}(F_x)$. We begin with the following estimate, see \cite[page 231]{DIbook}:
	\begin{lem}
		\label{doubling estimates}
		Let $f^H$ be as in \eqref{harmonification} and let $B \subset M \times \mathbb{R}$ be a ball of any radius, then
		\begin{align}
			||f^H||_{L^{\infty}(2B^{+})}\lesssim e^{CT} ||f^H||_{L^{\infty}(B^{+})}.\nonumber
		\end{align}
		for some $C>0$.
	\end{lem}
	We then have the following bound on $\mathcal{V}(F_x)$:
	\begin{lem}
		\label{doubling index lem}
		Let  $F_x$ be as in  \eqref{def of F_x}, then
		\begin{align}
			\mathcal{V}(F_x)\lesssim  T.  \nonumber
		\end{align}
	\end{lem}
	\begin{proof}
		Applying Proposition \ref{doubling index bound} to $f_{1/T}(\cdot)= f(T^{-1}\cdot)$ with $B= B_g(x,1)\times (0,1)$ (we tacitly assume that $T$ is sufficiently large so that $1/T\leq \eta/10$ with $\eta$ as in Proposition \ref{doubling index bound}), we obtain
		$$ 	\mathcal{V}\left(f_{1/T}, \frac{1}{2}B \cap M\right) \lesssim N (f^H_{1/T},2B).$$
		Since the $L^{\infty}$-norm is invariant under scaling, Lemma \ref{doubling estimates} gives $N (f^H_{1/T},2B)\lesssim T$. Therefore, Lemma \ref{doubling index lem} follows upon noticing that the definition of $F_x$  implies that $$   \mathcal{V}(F_x) \lesssim	\mathcal{V}\left(f_{1/T}, \frac{1}{2}B \cap M\right).$$
	\end{proof}
	However, using \cite[Theorem 5.3]{L1}, it is possible to improve on Lemma \ref{doubling index lem}. Indeed, we have the following:
	\begin{thm}[Logunov]
		\label{Logunov}
		Let $u$ be an harmonic function on $M \times \mathbb{R}$, and let $O\in M$. Then there exists constants $\kappa=\kappa(n)>0$, $N_0=N_0(M,g,n,O)$ and $R_0=R_0(M,g,n,O)$ such that for any cube $Q\subset B(O,R_0)$ the following holds: divide $Q$ into $B^n$ subcubes, then the number of subcubes with doubling index greater than $\max(N(u,Q)2^{-\log B/\log\log B}, N_0)$ is less than $O(B^{n-1-\kappa})$.
	\end{thm}
	In particular, by compactness, $M$ can be covered by finitely many cubes of radius $O(1)$; then, by covering each one of them by sub-cubes $B$ of radius $O(T^{-1})$ and applying Theorem \ref{Logunov} with $u=f^H$, in light of Lemma \ref{doubling estimates}, we obtain the following corollary:
	\begin{cor}
		\label{doubling index corollary}
		Let $F_x$ be as in section \ref{asymptotic gaussianity}. There exists some $\kappa=\kappa(n)>0$ and a subset $\tilde{A}\subset M$ of volume at most $O(T^{-1-\kappa})$, such that
		\begin{align}
			\mathcal{V}(F_x)\lesssim T 2^{-\log T/\log \log T}. \nonumber
		\end{align}
	for all $x\in M\backslash \tilde{A}.$
	\end{cor}

	\section{Anti-concentration}
	\label{anti-concentration section}
	The aim of this section is to show that $\mathcal{V}(F_x(\cdot, \omega))$ is uniformly integrable in $\omega\in \Omega$, that is we prove the following Proposition:
	\begin{prop}
		\label{anti-concentration long}
		Let $F_x$ be as in \eqref{def of F_x}, $v(T)$ be as in \eqref{defv}, and $g(t)=t\log t$. Suppose that  $v(T)\geq c_M T^2/\log T$, then there exists constants $C>0$, $\kappa=\kappa(n)>0$, independent of $T$, and a subset $\tilde{A}\subset M$ of volume at most $O(T^{-1-\kappa})$, such that
		\begin{align}
			\mathbb{E}[g(\mathcal{V}(F_x))]<C.  \nonumber
		\end{align}
		for all $x\in M\backslash \tilde{A}.$
	\end{prop}
	For the sake of clarity, we divide the proof of Proposition \ref{anti-concentration long}, into lemmas: we first estimate the probability that $F_x$ attains small values, and then show how this can be used to bound the nodal volume.
	\subsection{Small values of $f$}
	To estimate the probability of occurrence of small of $f$, we need the following lemma of Halasz, see for example \cite{H} and \cite[Lemma 6.2]{NV13}.
	\begin{lem}[Halasz' inequality]
		\label{Halas'z}
		Let $X$ be a real-valued random variable, and let $\psi(t)=\mathbb{E}[\exp(itX)]$ be its characteristic function. Then there exists some absolute constant $C>0$ such that
		\begin{align}
			\mathbb{P}(|X|\leq 1)\leq C \int_{|t|\leq 1} |\psi(t)|dt. \nonumber
		\end{align}
	\end{lem}
	We then prove the following Lemma:
	\begin{lem}
		\label{ small values of $f$}
		Let $f$ be as in \eqref{function} and $v=v(T)=c_M \rho(T)T^{n-1}$ as in \eqref{defv}. Then, uniformly for $x\in M$, we have
		\begin{align}
			\mathbb{P}\left( |f(x)|\leq \tau\right)\lesssim \tau + v(T)^{-1/2} \nonumber
		\end{align}
		where the constant implied in the $\lesssim$-notation is absolute.
	\end{lem}
	\begin{proof}
		By Lemma \ref{Halas'z}, we have
		\begin{align}
			\mathbb{P}(|f(x)|/\tau\leq 1)&\leq C \int_{|t|\leq 1} |\mathbb{E}[ \exp( it |f(x)|/\tau)]|dt \nonumber\\
			&= C \tau \int_{|t|\leq 1/\tau} |\mathbb{E}[ \exp( it |f(x)|)]|dt \nonumber \\
			&= C \tau \int_{|t|\leq 1/\tau} \left|\prod_{\lambda_i} \psi_{i}(t)\right| dt \label{14}
		\end{align}
		where $\psi_{i}(t)= \mathbb{E}[\exp(it a_{i}\phi_{i}(x)/v^{1/2})]$ and $C$ is as in Lemma \ref{Halas'z}.
		
		Now, let $K>1$ be some parameter to be chosen later, and observe that, in light of Lemma \ref{diagonal Weyl}, we have $$v^{-1}\sum_{\lambda_i}|\phi_{i}(x)|^2\lesssim2,$$ thus there are at most $O(v/K^2)$ $\lambda_i$'s such that $|\phi_{i}(x)|> K$. We claim that, if $|\phi_{i}(x)|\leq K$, there exists some $c>0$ such that
		\begin{align}
			\left|\psi_{i}(t)\right| \leq e^{-t^2v^{-1}}\label{comparison bound}
		\end{align}
		for all  $t\in [-c v^{1/2}/K, cv^{1/2}/K]$. Indeed, if $|\phi_{i}(x)|\leq K$, then  $a_{i} \phi_{i}(x)v^{-1/2}$ is a random variable with mean zero and second moment bounded by $K^2/v$. Thus, using the fact that $|\exp(iy)- (1+iy -2^{-1}y^2)|\leq \min( 6^{-1}y^3, y^2)$, we can write
		\begin{align}&\left|\psi_{i}(t)- \left( 1 - \frac{K^2}{2v}t^2\right)\right| \leq  \frac{K^2}{v}t^2 & \left| e^{-t^2v^{-1}} - \left( 1 - \frac{K^2}{2v}t^2\right)\right| \leq  \frac{K^2}{v}t^2, \label{6.1}
		\end{align}
		provided $t\in [-c v^{1/2}/K, cv^{1/2}/K]$ for some sufficiently small $c>0$. Therefore, \eqref{comparison bound} follows from \eqref{6.1}, choosing $c>0$ appropriately.

		Using the claim \eqref{comparison bound}, and since there are at least $ v\left(1 - O\left(K^{-2}\right)\right) $ values of $i$ such that $|\phi_{i}(x)|\leq K$, for all $t\in [-c v^{1/2}/K, cv^{1/2}/K]$, we have
		\begin{align}
			\left| \prod_{|\phi_{i}(x)|\leq K} \psi_{i}(t)  \right| \lesssim \exp \left( - t^2 \left( 1 - O(K^{-2})\right) \right). \nonumber
		\end{align}
		Thus, using the trivial bound $|\psi_{i}(t)|\leq 1$ if  $|\phi_{i}(x)|\geq K$, and picking $K=1000$,  for all $t\in [-c_1 v^{1/2}, c_1v^{1/2}]$, we obtain
		$$ \left| \prod_{i} \psi_{i} (t)\right| \lesssim   \exp \left( - c_2 t^2  \right),$$
		for some absolute constant $c_2>0.$ Hence, if $1/\tau\leq c_1 v^{1/2}$, that is $\tau \geq c_1^{-1}v^{-1/2}$,  the integral in  \eqref{14} is bounded, thus
		$$ 	\mathbb{P}(|f(x)|\leq \tau) \lesssim  \tau.$$ If $\tau \leq c_1^{-1}v^{-1/2}$, we have
		$$ \mathbb{P}(|f(x)|\leq \tau) \leq	\mathbb{P}(|f(x)|\leq c_1^{-1}v^{-1/2} ) \lesssim v^{-1/2}.$$
		This concludes the proof of Lemma \ref{ small values of $f$}.
	\end{proof}
	\subsection{A concentration inequality for the nodal volume}
	Lemma \ref{ small values of $f$} allows us to control the nodal volume of $F_x$, as follows:
	\begin{lem}
		\label{anti-concentration}
		Let $F_x$ be as in \eqref{def of F_x} and, let $K>1$ be some parameter. Then, for all $0< Q \lesssim v^{1/2}/K$,  we have
		\begin{align}
			\mathbb{P}\left(	\mathcal{V}(F_x)>Q\right)\lesssim\frac{1}{Q K} + \frac{(QK)^2}{e^{2Q}} \nonumber
		\end{align}
	\end{lem}
	\begin{proof}
		
		Given some $\tau>0$ , we write
		\begin{align}
			\mathbb{P}(	\mathcal{V}(F_x) > Q) = \mathbb{P}(	\mathcal{V}(F_x) > Q \hspace{3mm}\text{and} \hspace{3mm}|f(x)|<\tau ) + \mathbb{P}(	\mathcal{V}(F_x) > Q \hspace{3mm}\text{and} \hspace{3mm} |f(x)|\geq \tau). \label{10}
		\end{align}
		The first term on the r.h.s. of \eqref{10} can be bounded, via Lemma \eqref{ small values of $f$}, as
		\begin{align}
			\mathbb{P}(	\mathcal{V}(F_x) \geq Q \hspace{3mm}\text{and} \hspace{3mm}|f(x)|<\tau ) \leq \mathbb{P} ( |f(x)|\leq \tau )\lesssim \tau, \label{12}
		\end{align}
		provided $\tau \gtrsim v(T)^{-1/2}$.
		Let us now consider the second term on the r.h.s. of \eqref{10}, thus, we may assume that $|f(x)|>\tau$. Let $B= B(x,1/T)\times (-1/T,1/T)$,	by Proposition \ref{doubling index bound}, maintaining the same notation, we have
		\begin{align}
			\mathcal{V}(F_x)\lesssim   \log \frac{||f^H||_{L^{\infty}(4B^+)}}{||f^H||_{L^{\infty}(2B^+)}} \leq \log \frac{||f^H||_{L^{\infty}(4B^+)}}{|f(x)|}, \label{inequalit nodal volume}
		\end{align}
		thus
		\begin{align}
			\mathbb{P}(	\mathcal{V}(F_x) \geq Q \hspace{3mm}\text{and} \hspace{3mm} |f(x)| \geq \tau) \leq 	\mathbb{P}\left( ||f^H||_{L^{\infty}(2B^+)} > \tau e^Q\right). \label{13}
		\end{align}
		
		Now we claim the following:
		\begin{align}
			\label{claim 1}
			\mathbb{E}[||f^H||^2_{L^{\infty}(4B^+)}] =O(1).
		\end{align}
		Indeed, writing $f^H_{1/T}(\cdot)= f^H(T^{-1}\cdot)$ and $\tilde{B}= B(x,1)\times (-1,1)$, due to the fact that supremum norm is scale invariant, and upon using elliptic regularity, we have
		$$ ||f^H||_{L^{\infty}(4B^+)}\leq ||f^H_{1/T}||_{L^{\infty}(4\tilde{B})}\lesssim  ||f^H_{1/T}||_{L^2(6\tilde{B})}.$$
		Therefore, using the formula \eqref{f^h} and exchanging the expectation with the sum, we see that
		$$	\mathbb{E}[||f^H||^2_{L^{\infty}(4B^+)}]\lesssim  \mathbb{E}[||f^H_{1/T}||^2_{L^2(6\tilde{B})}] \lesssim v(T)^{-1}\int_{B(0,6)}\sum_{\lambda_i} |\phi_{i}(T^{-1}x)|^2 dx,$$
		thus \eqref{claim 1} follows from Lemma \ref{diagonal Weyl}.
		Using  \eqref{claim 1} together with Chebyshev's inequality, \eqref{12} and \eqref{13}, we obtain
		\begin{align} \nonumber
			\mathbb{P}(\mathcal{V}(F_x)\geq Q)\lesssim \tau + \frac{1}{\tau^2 e^{2Q}}
		\end{align}
		provided that $\tau \gtrsim v(T)^{-1/2}$. Hence, Lemma \ref{anti-concentration} follows by taking $\tau= 1/QK$, which gives $Q\leq v(T)^{1/2}/K$.
	\end{proof}
	\subsection{Concluding the proof of Proposition \ref{anti-concentration long}}
	We are finally ready to conclude the proof of Proposition \ref{anti-concentration long}:
	\begin{proof}[Proof of Proposition \ref{anti-concentration long}]
		By Corollary \ref{doubling index corollary}, we may assume that
		$$\mathcal{V}(F_x)\lesssim T 2^{-\log T/\log\log t}=:h(T).$$
		Therefore, letting $g= t \log t$ and integrating by parts, we have
		\begin{align} \label{17.1}
			\mathbb{E}[g(\mathcal{V}(F_x))]&= \int_0^{h(T)} g(t)d\mathbb{P}\left( \mathcal{V}(F_x)>t\right)=  \int_0^{h(T)}g'(t)\mathbb{P}(\mathcal{V}(F_x)> t)dt \nonumber \\
			&= \int_1 ^{h(T)} g'(t)\mathbb{P}(\mathcal{V}(F_x)> t)dt +O(1).
		\end{align}
		Let $K=K(t)>1$ be some parameter to be chosen later,	using Lemma \ref{anti-concentration}, the right hand  side of \eqref{17.1} is bounded by
		\begin{align} \label{6.3.1}
			\mathbb{E}[g(\mathcal{V}(F_x))]\lesssim \int_1^{h(T)}g'(t) \frac{1}{t K(t)} dt  + \int_1^{h(T)}g'(t) \frac{tK(t)}{e^{2t}} dt  +O(1),
		\end{align}
		provided $h(T)\leq v^{1/2}(T)/K(h(T))$. Taking $K(t)= \log (t)^3 $, we have $K(h(T))\gtrsim (\log T)^3$, thus, using \eqref{defv} and the assumption $v(T)\geq T^2/\log T$, we have
		$$ h(T)\lesssim \frac{T^{\frac{n-1}{2}}}{(\log T)^3}\leq v(T)^{1/2}/K(h(T)).$$
		Hence, Proposition \ref{anti-concentration long} follow from \eqref{6.3.1}.
	\end{proof}
	
	\section{Proof of Theorem \ref{main thm}}
	Before concluding the proof of Theorem \ref{main thm}, we need two standard lemmas.
	\subsection{Preliminary lemmas}
	For the reader's convenience, we include the proofs of the following two lemmas.
	\begin{lem}
		\label{locality}
		Let $F_x$ be as in \eqref{def of F_x} and $\omega_n$ be the volume of the ball of radius $1$ in $\mathbb{R}^n$ . Then, we have
		\begin{align}
			\mathcal{V}(f)= \frac{2^n T}{\omega_{n}}\left(1 +o_{T\rightarrow \infty}(1)\right)\int_M \mathcal{V}(F_x)dV_g(x) . \nonumber
		\end{align}
	\end{lem}
	
	\begin{proof}
		First, we observe that we may write
		\begin{align} \label{6.2.1}
			\mathcal{V}\left(f, B_g\left(x,\frac{1}{2T}\right)\right)= \int_{f^{-1}(0)} \mathds{1}_{y\in B_g(x,1/(2T))}d\mathcal{H}^{n-1}(y).
		\end{align}
		Then, integrating both sides of \eqref{6.2.1}  and using Fubini's Theorem, we have
		\begin{align}
			\label{6.2.4}
			\int_M   \mathcal{V}\left(f, B_g\left(x,\frac{1}{2T}\right)\right) dV_g(x)= \int_{f^{-1}(0)} \vol_g\left( B_g\left(y,\frac{1}{2T}\right)\right) d\mathcal{H}^{n-1}(y).
		\end{align}
		Now, we observe that, in light of the definition of $F_x$ in section \ref{asymptotic gaussianity}, we have
		$$  \mathcal{V}\left(f, B_g\left(x,\frac{1}{2T}\right)\right) \cdot T^{n-1}(1+o_{T\rightarrow \infty}(1))=\mathcal{V} (F_x).$$
		Thus, the LHS of \eqref{6.2.4} is
		\begin{align}
			\label{6.2.2}
			\int_M   \mathcal{V}\left(f, B_g\left(x,\frac{1}{2T}\right)\right) dV_g(x)= \frac{1}{T^{n-1}} \left(1 +o_{T\rightarrow \infty}(1)\right)\int_M \mathcal{V}(F_x)dx
		\end{align}
		Moreover, for all $y\in M$, we also have
		$$\vol_g\left( B_g\left(y,\frac{1}{2T}\right)\right) = \vol_{\mathbb{R}^n} \left(B\left(0,\frac{1}{2T}\right)\right)\left( 1 + O\left( T^{-1}\right)\right)= \frac{\omega_{n}}{(2T)^n}(1 +O(T^{-1})).$$
		Thus, the r.h.s. of \eqref{6.2.4} is
		\begin{align}
			\label{6.2.3} \int_{f^{-1}(0)} \vol_g\left( B_g\left(y,\frac{1}{2T}\right)\right) d\mathcal{H}^{n-1}(y)=  \frac{\omega_n}{(2T)^n}(1 + O( T^{-1})) \mathcal{V}(f)
		\end{align}
		Hence, Lemma \ref{locality} follows from  inserting \eqref{6.2.2} and \eqref{6.2.3} into \eqref{6.2.4}.
	\end{proof}
	
	\begin{lem}
		\label{Kac-Rice}
		Let $F_{\mu}$ be as in section \ref{asymptotic gaussianity}, and $\omega_n$ be the volume of the unit ball in $\mathbb{R}^n$. Then, we have
		$$ \mathbb{E}[ \mathcal{V}(F_{\mu})]= 2^{-n}\omega_n \left( \frac{4\pi}{n}\right)^{1/2} \frac{\Gamma\left( \frac{n+1}{2}\right)}{\Gamma \left(\frac{n}{2}\right)}.$$
	\end{lem}
	\begin{proof}
		Since the support of $\mu$, being the unit sphere, is not contained in an hyperplane,  the distribution of $(F_{\mu},\nabla F_{\mu})$ is non-degenerate. Thus, we may apply the Kac-Rice formula \cite[Theorem 6.1]{AW09}, to see that
		\begin{align}
			\label{7.1.1}
			\mathbb{E}[ \mathcal{V}(F_{\mu})] =\int_{2^{-1}B_0} \mathbb{E}[ |\nabla F_{\mu}(y)|| F_{\mu}(y)=0] \varphi_{F_{\mu}(y)}(0)dy,
		\end{align}
		where $\varphi_{F_{\mu}(y)}(0)$ is the density of $F_{\mu}(y)$ at the point $0$. Since $\mathbb{E}[|F_{\mu}(y)|^2]=1$, $\nabla F_{\mu}$ and $F_{\mu}$ are independent, and bearing in mind that $F_{\mu}$ is stationary (that is, $F_{\mu}(y)$ has the same distribution as $ F_{\mu}(0)$), we have
		\begin{align}\mathbb{E}[| \nabla F_{\mu}(y)|| F_{\mu}(y)=0] \varphi_{F_{\mu}(y)}(0)= \mathbb{E}[\nabla F_{\mu}(0)]\varphi_{F_{\mu}(0)}(0). \label{7.1}
		\end{align}
		The latter can be computed explicitly, see for example \cite[Proposition 4.1]{RW08}, to be
		\begin{align}
			\label{7.1.2} \mathbb{E}[|\nabla F_{\mu}(0)|]\varphi_{F_{\mu}(0)}(0)= \left( \frac{4\pi}{n}\right)^{1/2} \frac{\Gamma\left( \frac{n+1}{2}\right)}{\Gamma \left(\frac{n}{2}\right)}.
		\end{align}
		Hence, Lemma \ref{Kac-Rice} follows by inserting \eqref{7.1.2} into \eqref{7.1.1} via \eqref{7.1}.
	\end{proof}
	\subsection{Concluding the proof of Theorem \ref{main thm}}
	
	\begin{proof}[Proof of Theorem \ref{main thm}]
		Thanks to Lemma \ref{locality} and Fubini's Theorem, we have
		\begin{align}
			\label{7.2.1}\mathbb{E}[\mathcal{V}(f)]= \frac{2^n T}{\omega_n} \left(1 +o_{T\rightarrow \infty}(1))\right)\int_M \mathbb{E}[\mathcal{V}(F_x)]dV_g(x)
		\end{align}
		Thanks to Propositions \ref{asymptotic Gaussianity}, Proposition  \ref{anti-concentration long}, whose assumptions are satisfied thanks to the assumptions of Theorem \ref{main thm}, and the Dominated Convergence Theorem, we have
		\begin{align}
			\label{7.2.2}
			\mathbb{E}[\mathcal{V}(F_x)]= \mathbb{E}[\mathcal{V}(F_{\mu})](1+o_{T\rightarrow \infty}(1)),
		\end{align}
		uniformly for all $x\in M$ outside a subset $A\subset M$ of volume at most $O(\log T/T)$.
		Moreover, thanks to Corollary \ref{doubling index corollary}, we have
		$$ \mathcal{V}(F_x) \lesssim T 2^{-\log T/\log\log T},$$
uniformly		for all $x$ outside a subset $\tilde{A}\subset M$ of volume at most $O(T^{-1-\kappa})$ with $\kappa>0$ independent of $T$.

		Thus, using \eqref{7.2.2} and Lemma \ref{doubling index lem}, we may rewrite the integral in \eqref{7.2.1} as
		\begin{align}
			\label{7.2.3}
			\int_M \mathbb{E}[\mathcal{V}(F_x)]dV_g(x)&= \int_{M\backslash A} \mathbb{E}[\mathcal{V}(F_x)]dV_g(x) + \int_{A\cap \tilde{A}} \mathbb{E}[\mathcal{V}(F_x)]dV_g(x)  + \int_{A\backslash \tilde{A}} \mathbb{E}[\mathcal{V}(F_x)]dV_g(x) \nonumber \\
			&= \vol(M\backslash A)\mathbb{E}[\mathcal{V}(F_{\mu})](1+o_{T\rightarrow \infty}(1)) +o_{T\rightarrow \infty}(1) \nonumber \\
			&= \vol(M)\mathbb{E}[\mathcal{V}(F_{\mu})](1+o_{T\rightarrow \infty}(1)) +o_{T\rightarrow \infty}(1).
		\end{align}
		Inserting \eqref{7.2.3} into \eqref{7.2.1} and using Lemma \ref{Kac-Rice} we obtain
		\begin{align} \nonumber
			\mathbb{E}[\mathcal{V}(f)]&= \frac{2^n}{\omega_n}\vol(M)\mathbb{E}[\mathcal{V}(F_{\mu})] \cdot  \left(T +o_{T\rightarrow \infty}(T)\right) \\
			&= \vol(M)\left(\frac{4\pi}{n} \right)^{1/2}  \frac{\Gamma\left( \frac{n+1}{2} \right)}{\Gamma\left( \frac{n}{2}  \right)}T
			+o_{T\rightarrow\infty}(T),  \nonumber
		\end{align}
		as required.
	\end{proof}
	
	\appendix
	
	\section{Important case: random non-Gaussian spherical harmonics}
	\label{sec:sphere}
	
	The sequence of Laplace eigenvalues on $\mathbb{S}^2$, the two dimensional round unit sphere, is given by $\{\ell(\ell+1)\}$, where $\ell>0$ runs through the integers. Each eigenvalue has multiplicity $2\ell+1$ and the corresponding  Laplace eigenfunctions  are the restrictions of homogeneous harmonic polynomials of degree  $\ell$ to $\mathbb{S}^2$, also known as spherical harmonics.
	
	There exists a \textquotedblleft distinguished\textquotedblright orthonormal base for the space of spherical harmonics of degree $\ell$, see for example \cite[Theorem 4.9]{Zbook},
	\begin{align}
		\label{particular base}
		Y_{\ell,m}(\theta, \varphi)= \left( (2\ell +1) \frac{(\ell-m)!}{(\ell+m)!}\right)^{1/2}P_{\ell}^m(\cos(\theta))e^{im\varphi},
	\end{align}
	where $m=-\ell,...,\ell$, $(\theta, \varphi)$ are polar-coordinates on $\mathbb{S}^2$ and $P_{\ell}^m(x)$ are the associated Legendre polynomials. Therefore, the band-limited functions on $\mathbb{S}^2$ with spectral parameter $T=\ell(\ell+1)$ and width of the energy window$\rho(T)=1$, such as $f_{\ell}$ in \eqref{eq:fl spher har}, can be expressed as a linear combinations of the functions in \eqref{particular base}.
	
	We observe that, under the notation \eqref{def of F_x} of section \ref{local Weyl's law}, thanks to the Funk-Hecke formula and the Hilb’s asymptotics for Legendre polynomials (see e.g. \cite[Claim 2.7 and Claim 2.9]{NS09} and references therein), one has
	$$ \mathbb{E}[F_{x,\ell}(y)F_{x,\ell}(y')]= J_0(|y-y'|) +o_{\ell \rightarrow \infty}(1),$$
	and we can also differentiate both sides any arbitrary finite number of times.
	
	In this section, we prove the following theorem:
	
	\begin{thm}
		\label{thm sphere}
		Let
		$$f_{\ell}(x)= \sum_{m=-\ell}^{\ell}a_mY_{\ell,m}(x),$$
		where the $Y_{\ell,m}$ are given in \eqref{particular base} and assume that $a_{\ell,m}$ are i.i.d. $\pm 1$ Bernoulli random variables. Then
		$$	\E[\mathcal{V}(f_{\ell})] = \frac{2}{\sqrt{2}}\pi \ell +o_{\ell \rightarrow \infty}(\ell). $$
	\end{thm}
	
	The main crux of extending the proof of Theorem \ref{main thm} to the $2$d-case is the failure of Proposition \ref{anti-concentration long} in $2$d. However, in the setting of Theorem \ref{thm sphere}, it is possible to overcome the said obstacle by invoking the following result due to Nazarov, Nishry and Sodin \cite[Corollary 2.2]{NNS13} on the log-integrability of random Fourier series with Bernoulli coefficients:
	\begin{lem}[{~\cite[Corollary 2.2]{NNS13}}]
		\label{log integrability}	
		Let $g:\mathbb{S}^{1}\times \Omega \rightarrow \mathbb{\mathbb{C}}$ be a random Fourier series given by
		$$g(\varphi)= \sum_{m=-\infty}^{\infty} a_m b_m e(m\varphi),$$
		where $\{a_m\}_{m\in\Z}$ are i.i.d. $\pm 1$ Bernoulli random variables defined on the probability space $\Omega$ and
		$\{b_{m}\}_{m\in\Z}$ is a sequence of (deterministic) complex-valued coefficients
		so that $$||g||_{L^2(\mathbb{S}^{1}\times \Omega)}=1.$$
		Then, for every integer $p>0$, there exists a constant $C=C(p)>0$ such that
		$$\int\limits_{\mathbb{S}^{1}\times \Omega} \left|\log |g| \right|^p d\nu \leq C,$$
		where $\nu$ is the product measure of the Lebesgue measure on $\mathbb{S}^{1}$ and $\mathbb{P}(\cdot)$ on $\Omega$.
	\end{lem}
	
	In fact, Lemma \ref{log integrability} implies the following:
	\begin{cor}
		\label{log-integrability corollary}
		Let $f_{\ell}$ be as in Theorem \ref{thm sphere}, $x\in\mathbb{S}^{2}$, and $F_x= F_{x,\ell}$ the scaled version of $f_{\ell}$
		as in \eqref{def of F_x}. Then there exists a constant $C>0$ such that
		$$ \frac{1}{4\pi}\int\limits_{\mathbb{S}^2 \times \Omega} \mathcal{V}(F_x)^2d\sigma < C,$$
		where $\sigma$ is the product measure of the Lebesgue measure on $\mathbb{S}^2$ and $\mathbb{P}(\cdot)$ on $\Omega$.
	\end{cor}
	\begin{proof}
		The proof of Corollary \ref{log-integrability corollary} is very similar to the proof of Lemma \ref{anti-concentration}, so we omit some details while retaining same notation. Thanks to the inequality \eqref{inequalit nodal volume} and the fact that $(X+Y)^2\lesssim X^2+Y^2,$
		we have
		\begin{align}
			\label{9.1.1}
			\int\limits_{\mathbb{S}^2 \times \Omega} \mathcal{V}(F_x)^2 d\sigma\lesssim  \int\limits_{\mathbb{S}^2 \times \Omega}   \left(\log \frac{||f^H_{\ell}||_{L^{\infty}(4B^+)}}{||f^H_{\ell}||_{L^{\infty}(2B^+)}}\right)^2 d\sigma \lesssim O(1)+\int\limits_{\mathbb{S}^2 \times \Omega}   |\log|f_{\ell}||^2 d\sigma,
		\end{align}
		where $\sigma$ is the product measure on $\mathbb{S}^2\times \Omega$. Now, we wish to use Lemma \ref{log integrability} to estimate the second term on the r.h.s. of \eqref{9.1.1}.

		First, we observe that, writing $f_{\ell}=f_{\ell}(\varphi,\theta)$, where $(\varphi,\theta)$ are spherical coordinates, and using Fubini's Theorem, we have
		\begin{align}
			\label{9.1.2}
			\int\limits_{\mathbb{S}^2 \times \Omega}   |\log|f_{\ell}||^2 d\sigma\leq \int\limits_0^{\pi} d\theta \int\limits_{\mathbb{S}^{1} \times \Omega}  |\log|f_{\ell}(\varphi,\theta)||^2 d\nu,
		\end{align}
		where $d\nu= d\varphi \otimes d\mathbb{P}$ is the product measure on $\mathbb{S}^{1} \times \Omega$.
		Writing $f_{\ell}(x)=f_{\ell}(\varphi, \theta)$ and bearing in mind \eqref{particular base}, we have
		$$f_{\ell}(\varphi,\theta)= \sum_{m=-\ell}^{\ell } a_{\ell,m} \left( \frac{(\ell-m)!}{(\ell+m)!}\right)^{1/2}P_{\ell}^m(\cos(\theta))e^{im\varphi}.$$
		Therefore, fixing $\theta$, thanks to the normalisation $\mathbb{E}[|f_{\ell}(\varphi,\theta)|^2]=1$, we may apply Lemma \ref{log integrability} with $b_m= \frac{(\ell-m)!}{(\ell+m)!}^{1/2}P_{\ell}^m(\cos(\theta))$, to see that
		$$\int\limits_{\mathbb{S}^{1} \times \Omega}  |\log|f_{\ell}||^2 d\nu \leq C$$
		for some absolute constant $C>0$. Hence, thanks to \eqref{9.1.2}, the r.h.s. of \eqref{9.1.1} is bounded and this concludes the proof of Corollary \ref{log-integrability corollary}.
	\end{proof}
	
	We are finally in the position to prove Theorem \ref{thm sphere}:
	\begin{proof}[Proof of Theorem \ref{thm sphere}]
		As in the proof of Theorem \ref{main thm}, Lemma \ref{locality}, Fubini's Theorem and Corollary \ref{doubling index lem} give
		\begin{align}
			\label{f1}	\mathbb{E}[\mathcal{V}(f)]=  \frac{4}{\pi}\ell\cdot \left(1 +o_{\ell \rightarrow \infty}(1)\right)\int\limits_{\mathbb{S}^2\times \Omega} \mathcal{V}(F_x)d\sigma,
		\end{align}
		where $\sigma$ is the product measure on $\mathbb{S}^2\times \Omega$. Now, Proposition \ref{asymptotic Gaussianity} implies that
		$$ \mathcal{V}(F_x) \overset{d}{\longrightarrow} \mathcal{V}(F_{\mu}) \hspace{8mm}\ell \rightarrow \infty,$$
		in the product space $\mathbb{S}^2\times \Omega$. Thus, Corollary \ref{log-integrability corollary} together with the Dominated Convergence Theorem  give
		$$ \int\limits_{\mathbb{S}^2\times \Omega} \mathcal{V}(F_x)d\sigma \longrightarrow \int\limits_{\mathbb{S}^2\times \Omega} \mathcal{V}(F_\mu)d\sigma \hspace{8mm}\ell \rightarrow \infty.$$
		Therefore, another application of Fubini's Theorem together with \eqref{f1} give
		\begin{align}
			\label{f2}	\mathbb{E}[\mathcal{V}(f)]= \frac{4}{\pi} \ell\cdot \left(1 +o_{\ell\rightarrow \infty}(1)\right)\int\limits_{\mathbb{S}^2} \mathbb{E}[ \mathcal{V}(F_\mu)]dx
		\end{align}
		Finally, the right hand side of \eqref{f2} can be computed via Lemma \ref{Kac-Rice} and Theorem \ref{thm sphere} follows.
	\end{proof}
	
	\bibliographystyle{siam}
	\bibliography{ExpectationNV}
\end{document}